\documentclass[acmsmall]{acmart}

\usepackage{xspace,cleveref,stmaryrd,microtype}

\usepackage{todonotes}
\setuptodonotes{inline}

\newtheorem{theorem}{Theorem}
\newtheorem{property}[theorem]{Property}
\newtheorem{definition}[theorem]{Definition}

\newcommand{\name}{Rival\xspace}
\newcommand{\herbie}{Herbie\xspace}

\newcommand{\F}[1]{\text{\textsf{#1}}}
\newcommand{\K}[1]{\text{\textbf{#1}}}
\newenvironment{block}{\left|\begin{array}{l}}{\end{array}\right.}

\newcommand{\SearchIterations}{14\xspace}

\renewcommand{\ll}{\llbracket}
\newcommand{\rr}{\rrbracket}
\newcommand{\mathN}{\text{\texttt{N}}\xspace}

\newcommand{\percentpointssearchsaves}{74.6\%\xspace}

\newcommand{\pointssearchsavesfromrivaldomainerror}{9142\xspace}
\newcommand{\pointssearchsavesfromrivalunsamplable}{668\xspace}

\newcommand{\overallmathematicacrash}{16\xspace}
\newcommand{\overallmathematicacrashtimeoutormemory}{64\xspace}
\newcommand{\overallallpoints}{126720\xspace}

\newcommand{\overallmathematicaunsamplable}{7617\xspace}
\newcommand{\overallrivalunsamplable}{746\xspace}

\newcommand{\overallmathematicatimesrivalunknown}{1.4$\times$\xspace}

\newcommand{\TotalHardPoints}{19895\xspace}
\newcommand{\timesworsemathematicaoverallunsamplable}{10.2$\times$\xspace}
\newcommand{\rivalhardsamplesorerror}{18824\xspace}%
\newcommand{\mathematicahardsamplesorerror}{11744\xspace}
\newcommand{\percentrivalhardbetter}{60.3\%\xspace}
\newcommand{\rivalhardsamples}{5872\xspace}
\newcommand{\mathematicahardsamples}{24\xspace}
\newcommand{\rivalharderror}{12952\xspace}
\newcommand{\mathematicaharderror}{11720\xspace}
\newcommand{\rivalhardunresolved}{1071\xspace}
\newcommand{\mathematicahardunresolved}{8151\xspace}
\newcommand{\totalrivalsamplesorerror}{5877\xspace}%
\newcommand{\totalmathematicasamplesorerror}{26\xspace}

\newcommand{\totalrivalsamplablemathematicaunsamplable}{5830\xspace}

\newcommand{\timesrivalfastermathematica}{9.87$\times$\xspace}
\newcommand{\timesrivalfastermathematicawithoutcrashestimeouts}{9.03$\times$\xspace}

\newcommand{\percentfasterthanmpfi}{(19.2\%)\xspace}
\newcommand{\TotalHerbieBenchmarks}{481\xspace}
\newcommand{\TotalFPBenchBenchmarks}{126\xspace}
\newcommand{\TotalDemoBenchmarks}{4888\xspace}
\newcommand{\MinHerbieBenchmarkVariables}{0\xspace}
\newcommand{\MaxHerbieBenchmarkVariables}{16\xspace}
\newcommand{\MinHerbieBenchmarkOperations}{0\xspace}
\newcommand{\MaxHerbieBenchmarkOperations}{89\xspace}
\newcommand{\TotalWithConditionals}{4\xspace}
\newcommand{\TotalWithPrecondition}{47\xspace}

\newcommand{\PercentFailPreconditionOfBenchmarksWithPrecondition}{0.1\%\xspace}

\newcommand{\NumBenchmarksWithUnsamplablePointsSearchDisabled}{21\xspace}%
\newcommand{\PercentUnsamplablePointsUndetectedSearchDisabled}{19.0\%\xspace}
\newcommand{\SearchInvalidSampleReduction}{68.6\%\xspace}
\newcommand{\PercentSpaceT}{54.9\%\xspace}

\newcommand{\PercentSpaceF}{22.8\%\xspace}
\newcommand{\GuaranteeSampleChance}{71.1\%\xspace}
\newcommand{\PercentUnsamplableDetected}{81.0\%\xspace}
\newcommand{\NumberOfSearchIterations}{15\xspace}
\newcommand{\HammingPercentUnsamplableDetectedSearchDisabled}{100.0\%\xspace}
\newcommand{\FPBenchValidProbability}{93.8\%\xspace}
\newcommand{\FPBenchF}{72.6\%\xspace}

\begin{document}

\title{An Interval Arithmetic for Robust Error Estimation}

\author{Oliver Flatt}
\affiliation{
  \institution{University of Utah}
  \streetaddress{201 Presidents' Cir}
  \city{Salt Lake City}
  \state{UT}
  \country{USA}
}
\email{o.flatt@utah.edu}

\author{Pavel Panchekha}
\affiliation{
  \institution{University of Utah}
  \streetaddress{201 Presidents' Cir}
  \city{Salt Lake City}
  \state{UT}
  \country{USA}
  \postcode{84112-0090}
}
\email{pavpan@cs.utah.edu}

\begin{abstract}
  Interval arithmetic is a simple way
    to compute a mathematical expression
    to an arbitrary accuracy,
    widely used for verifying floating-point computations.
  Yet this simplicity belies challenges.
  Some inputs
    violate preconditions or cause domain errors.
  Others cause the algorithm to enter an infinite loop
    and fail to compute a ground truth.
  Plus, finding valid inputs is itself a challenge
    when invalid and unsamplable points
    make up the vast majority of the input space.
  These issues can make interval arithmetic
    brittle and temperamental.
  
  This paper introduces three extensions to interval arithmetic
    to address these challenges.
  \textit{Error intervals}
    express rich notions of input validity
    and indicate whether
    all or some points in an interval
    violate implicit or explicit preconditions.
  \textit{Movability flags} detect futile recomputations
    and prevent timeouts
    by indicating whether a higher-precision recomputation
    will yield a more accurate result.
  And \textit{input search} restricts sampling
    to valid, samplable points,
    so they are easier to find.
  We compare these extensions  to the
    state-of-the-art technical computing software Mathematica, 
    and demonstrate that our extensions are able
    to resolve \percentrivalhardbetter more challenging inputs,
    return \timesworsemathematicaoverallunsamplable fewer completely indeterminate results,
    and avoid \overallmathematicacrashtimeoutormemory cases of fatal error.
  
\end{abstract}

\ccsdesc[300]{Mathematics of computing~Arbitrary-precision arithmetic}
\ccsdesc[100]{Computing methodologies~Representation of exact numbers}

\keywords{Interval arithmetic, floating point error, overflow, computed reals}

\maketitle

\section{Introduction}

Floating-point arithmetic
  in scientific, financial, and engineering applications
  suffers from rounding error:
  the results of a floating-point expression can differ starkly
  from those of the analogous mathematical expression~%
  \cite{berkeley00-needle-like,kahan-java-hurts,cse14-practical-fp,book87-nmse}.
Rounding error has been responsible for
  scientific retractions~\cite{num-issues-in-stat,num-replication},
  mispriced financial indices~\cite{distort-stock,euro-rounding},
  miscounted votes~\cite{round-elections},
  and wartime casualties~\cite{patriot}.
Floating-point code thus needs careful validation;
  sound upper bounds~\cite{daisy,fptaylor,satire},
  semidefinite optimization~\cite{real2float},
  input generation~\cite{s3fp,fpgen},
  and statistical methods~\cite{perturbing-numerical}
  have all been proposed for such validation.
Sampling-based error estimation,
  which can estimate typical, not worst-case, error
  is especially widely used%
  ~\cite{precimonious,herbie,herbgrind,api-for-real-numbers}.

Consider the core task of computing the error
  for a floating-point computation at a given point.
This requires computing an error-free ground-truth for a mathematical computation.
Interval arithmetic%
  ~\cite{api-for-real-numbers,cse14-practical-fp,boost-ivals,moore-ivals,interval-constructive}
  and its many variations%
  ~\cite{computable-reals,boehm-idea,applied-interval-analysis}
 is the traditional approach to this problem.
Interval arithmetic provides sound bounds on the error-free ground truth,
  and allows for refining those bounds by recomputing at higher precision,
  until the ground truth can be estimated to any accuracy.
Computing the ground truth this way on a large number of sampled points gives
  a good estimate of the floating-point error of a computation.

In practice, however, even state of the art implementations
  struggle on invalid or particularly challenging input points.
The bounds are meaningless for invalid inputs;
  the recomputation process may enter an infinite loop;
  and valid inputs may be too hard to find.
Since error estimation requires sampling a large number of points,
  these challenges arise frequently.
As a result, even state of the art interval arithmetic implementations
  enter infinite loops, give up too early, or in some exceptional cases
  even experience memory exhaustion and fatal errors.

We identify three particular challenges
  for sampling-based error estimation:
  invalid input points,
  futile recomputations,
  and low measure.
\textit{Invalid input points} violate explicit or implicit preconditions
  and cannot meaningfully be evaluated to a ground truth value;
  consider the square root of a negative number.
Validity must thus be soundly tracked and checked,
  especially for inputs at the borderline between valid and invalid.
\textit{Futile recomputation} afflicts some valid inputs:
  interval arithmetic recomputes with ever-higher precision
  but without converging on a ground truth answer.
These recomputations must thus be soundly cut off early,
  warning the user that a ground truth value cannot be computed.
\textit{Low measure} means valid inputs
  make up a small proportion of the input space,
  making it difficult for applications to find sufficiently many.
Regions of valid input points must thus be efficiently and soundly identified.
No general-purpose strategy exists to address these independent issues;
  yet all three must be addressed for interval arithmetic to be robust.

This paper proposes three improvements
  to interval arithmetic that address these issues.
First, \textit{error intervals} determine whether
  all or some of the points in an interval
  violate implicit or explicit preconditions.
Second, \textit{movability flags}
  detect most input values for which recomputation will not converge,
  warning the user and avoiding futile recomputation.
Third, \textit{input search}
  discards invalid regions of the input space,
  focusing sampling on valid points.
Combined, our improvements
  soundly address the issues identified above
  and make interval analysis more robust.

We implement these improvements
  in the new interval arithmetic library called \name
  and compare \name to Mathematica's analogous \mathN function
  on a benchmark suite of \TotalHerbieBenchmarks~floating-point expressions.
\name produces results in dramatically more cases than Mathematica:
  \name is able to resolve \rivalhardsamplesorerror inputs,
  while Mathematica is only able to resolve \mathematicahardsamplesorerror (\percentrivalhardbetter better).
In only \overallrivalunsamplable cases does \name return a completely indeterminate result;
  Mathematica, however, does so in \overallmathematicaunsamplable cases (\timesworsemathematicaoverallunsamplable worse),
  and among those cases sometimes
  enters an infinite loop, runs out of memory, or hard crashes (\overallmathematicacrashtimeoutormemory cases).
Furthermore, \name's input search
  saves \percentpointssearchsaves of invalid points from being sampled in
  the first place.
An ablation study
  shows that error intervals, movability flags, and input search
  effectively handle even difficult inputs across a range of domains.

\medskip
\noindent
This paper contributes three extensions to interval arithmetic:
\begin{itemize}
\item \emph{Error intervals} to track domain errors and rich notions
  of input validity (\Cref{sec:impl});
\item \emph{Movability flags} to detect when recomputing in higher precision is futile (\Cref{sec:movable});
\item \emph{Input search} to uniformly sample valid inputs with high probability (\Cref{sec:search}).
\end{itemize}
\Cref{sec:eval,sec:analysis} demonstrate that these extensions
  successfully address the issues of invalid inputs,
  futile recomputations, and low measure.
\Cref{sec:discussion} discusses our experience
  deploying these extensions to users.

\section{Overview}

Imagine a simulation that evaluates $x^y / (x^y + 2)$.
How accurate is this expression?
This section describes
  how \name's extensions to interval arithmetic
  help answer the question.
For ease of exposition,
  this section uses two-digit decimal arithmetic as a target precision,
  with ordinary floating-point as a higher precision arithmetic;
  targetting single- or double-precision floating-point
  with arbitrary-precision libraries sees analogous issues.

\subsection{Interal Arithmetic}

Consider an input like $(x, y) = (3.0, 1.1)$.
At this input point, the expression $x^y / (x^y + 2)$
  can be directly computed with two-digit decimal arithmetic:
  ${3.0}^{1.1} = 3.34\dotsb$, which rounds to $3.3$;
  then $3.3 + 2 = 5.30\dotsb$, which rounds to $5.3$;
  and finally $3.3 / 5.3 = 0.622\dotsb$, which rounds to $0.62$.
How accurate is this result?
To find out, we must compare the computed result
  with the true value of this expression.

One might estimate that true value
  using higher-precision computation.
In single-precision floating-point,
  the input evaluates to $0.62605417$,
  which rounds to the two-digit decimal $0.63$.
Two-digit decimal evaluation therefore has error:
  it produces $0.62$ instead of the correct $0.63$.
Unfortunately, single-precision floating-point itself has rounding error.
After all, in double-precision arithmetic the computed value is
  not $0.626054\underline{17}$ but $0.626054\underline{25\dotsb}$,
  and even higher precision could produce some third value.
So this method of computing the true answer
  is suspect: how much precision is enough?

Interval arithmetic answers this question
  using \textit{intervals}, written $[a, b]$, which represent
  any real value between $a$ and $b$ (inclusive).
Ordinary mathematical operations are then extended to intervals,
  the two ends always rounded outward
  to ensure that the resulting interval always contains
  the mathematically exact value.
Returning to the running example,
  the interval evaluation of $x^y$
  with $(x, y) = (3.0, 1.1)$ is $[3.0, 3.0]^{[1.1, 1.1]} = [3.3, 3.4]$.
Next, $[3.3, 3.4] + [2.0, 2.0] = [5.3, 5.4]$.
Finally, $[3.3, 3.4] / [5.3, 5.4] = [0.61, 0.65]$,
  with those two endpoints computed via
  $3.3 / 5.4 = 0.61$ and $3.4 / 5.3 = 0.65$.%
\footnote{
  Division yields its smallest output
  with a small numerator and large denominator;
  more generally, evaluating a mathematical operation on intervals
  means solving an optimization problem,
  identifying the maximum and minimum values
  a function can take on over a range of inputs.}
Therefore, the true value is between $0.61$ and $0.65$.
A similar interval computation can be done with any precision;
  with single-precision floating-point,
  the resulting interval is
  $[0.626054\underline{17}, 0.626054\underline{41}]$.

Importantly, this interval is much narrower
  and both endpoints round to $0.63$ in two-digit decimals,
  meaning that the true value (which must be in the interval) does the same.
These \textit{one-value} interval therefore provide a way
  to evaluate the true value of the expression,
  to a target precision, using a higher but finite precision.
If even single-precision floating-point resulted in too wide an interval,
  double- or higher-precision arithmetic could be tried
  until a one-value interval is found.

Interval arithmetic thus provides a fast way
  to compute a ground-truth value for an expression on an input point.
This algorithm can be used for a range of purposes.
In sampling-based error estimation,
  a ground truth is computed for many randomly-sampled inputs
  to a floating-point expression,
  and the actual floating-point answer is compared to the ground truth
  to determine how accurate the expression is.
However, while this works well on most inputs,
  on challenging expressions
  it can throw errors, enter infinite loops, or fail to find valid inputs.

\subsection{Invalid Input Points}

Not all inputs to $x^y / (x^y + 2)$ are as conceptually simple
  as $(3, 1.1)$; consider $(x, y) = (-1.1, 7)$.
For this input, the denominator $x^y + 2$ evaluates to the interval
  $[0.0, 0.1]$ in two-digit decimals,
  meaning a division by zero error is possible.
But single-precision evaluation shows
  that the error is a mirage:
  $[-1.1, -1.1]^{[7,7]} + 2$ yields the interval
  $[0.051284\underline{1}, 0.051284\underline{3}]$
  for the denominator,
  proving that division by zero does not occur.
More broadly, in interval arithmetic,
  an interval may contain both valid and invalid inputs for an operator,
  in which case an error can be possible without being guaranteed.

\textit{Error intervals} formalize this reasoning.
Each interval is augmented with
  a boolean interval $[g, p]$,
  where the $g$ tracks whether an error is guaranteed
  and the $p$ tracks whether it is possible;
  the interval $[\top, \bot]$ is illegal,
  since a guaranteed error implies a possible error.
For our running example with $(x, y) = (-1.1, 7)$,
  the error interval is $[\bot, \top]$ in two-digit decimal arithmetic,
  indicating that error is not guaranteed but is possible,
  and $[\bot, \bot]$ in single-precision floating-point,
  guaranteeing that no error occurs.
Importantly,
  error intervals can handle borderline cases with recomputation.
For this point, two-digit decimal arithmetic is insufficient
  to determine whether the input point is valid,
  but single-precision is.
If even single-precision weren't, double or higher precision could be tried.

Boolean intervals are a type of three-valued logic which
  can also address expressions with explicit preconditions
  or with conditionals that branch on floating-point comparisons.
In each case a boolean interval
  indicates whether a condition can, or must, be true,
  and multiple boolean intervals
  can be combined with the standard boolean operators.
For example, for an expression $E(x)$ with precondition $\sin(x) > 0$,
  the valid inputs $x$ are those
  where $\sin([x, x]) > 0 \land \neg \F{err}(E([x, x])) = [\top, \top]$,
  where $\F{err}(E)$ refers to the error interval of $E$.
The key for interval arithmetic is that error and boolean intervals
  allow expressing rich notions of input validity
  that integrate with recomputation
  and provide sound, error-free guarantees.

\subsection{Futile Recomputation}

Recomputation guarantees soundness but not speed.
To the contrary: repeated, futile recomputations
  present a significant practical problem,
  most commonly due to overflow.

Consider our running example $x^y / (x^y+2)$,
  but with the extreme input $(x, y) = (10^{10}, 10^{10})$.
Now $x^y$ is huge---too large for double-precision
  and even for common arbitrary-precision libraries.
So in interval arithmetic, it evaluates to $[\Omega, \infty]$
  for some largest finite representable value $\Omega$.
Since $\Omega$ is already the largest representable value,
  $[\Omega, \infty] + 2$ results in $[\Omega, \infty]$ again,
  and since the division $[\Omega, \infty] / [\Omega, \infty]$
  can produce a value as small as $\Omega/\infty = 0$
  or as large as $\infty/\Omega = \infty$,
  the final interval is $[0, \infty]$.
This interval is too wide to be helpful,
  and since the problem is overflow, not rounding,
  recomputing at a higher precision will not help.
But we need to prevent recomputation
  from taking futile recourse to higher and higher precisions
  until an error like timeout, memory exhaustion, or worse occurs.

\emph{Movability flags} detect such futile recomputations.
Movability flags are a pair of boolean flags
  added to every interval,
  which track whether overflow influences each endpoint.
When $x^y$ overflows, for example,
  the infinite endpoint in $[\Omega, \infty]$ is marked ``immovable'',
  meaning that recomputing it at higher precision
  would yield the same result.
These movability flags are then
  propagated through interval computations:
  $[\Omega, \infty] + 2$ also has an immovable infinity,
  and in the final result $[0, \infty]$ both endpoints are immovable.
This guarantees that higher-precision evaluation
  would produce the same uselessly-wide interval
  and cuts off further recomputation.
Importantly, immovable endpoints distinguish
  harmful overflows that produce timeouts, like these,
  from benign overflows that still produce accurate results,
  like in $1/e^x$.

Movability flags are sound:
  if an expression evaluates to a wide immovable interval,
  recomputation cannot compute an error-free ground truth.
This allows distinguishing harmful overflows
  from difficult expressions that require higher precision to resolve.
Movability flags are not complete---%
  the problem is likely undecidable~\cite{api-for-real-numbers}---%
  but are accurate enough to dramatically cut the number of cases
  of futile recomputation.

\subsection{Low Measure}

Many applications of interval arithmetic,
  such as for error estimation of floating-point expressions,
  require sampling a large number of valid input points,
  usually via rejection sampling
  (repeatedly sampling points and discarding invalid ones).
Often a specific distribution, like uniform sampling, is required as well.
But for many mathematical expressions the valid, samplable points
  have \textit{low measure}: they are a tiny subset of the entire input space.
Consider $\sin^{-1}(x+2007)$:
  the only valid inputs are in $[-2008, -2006]$,
  a range that contains approximately $0.0001\%$
  of double-precision floating-point values.
Rejection sampling would yield almost exclusively invalid points!
Since restrictions on multiple variables combine multiplicatively,
  the chance a sample is valid decreases further for more variables.

\emph{Input search} addresses this issue
  by identifying and ignoring regions of the input space
  where all inputs are invalid or unsamplable.
Input search recursively subdivides the sample space
  into axis-aligned \textit{hyperrectangles},
  which assign each variable an interval.
For each hyperrectangle,
  input search then uses error intervals and movability flags
  to determine whether points in that interval are valid and samplable.
In a form of branch-and-bound search,
  input search then discards hyperrectangles that return $[\bot, \bot]$
  and subdivides those that return $[\bot, \top]$.
For example, for $\sin^{-1}(x + 2007)$,
  the interval $[0, \infty]$ is discarded
  but the interval $[-\infty, 0]$,
  whose error interval is $[\bot, \top]$,
  is subdivided and searched further.

Input search leverages error intervals and movability flags
  to ensure that difficult regions of the input space are handled soundly.
This ensures that input regions that contain valid inputs
  are never discarded,
  and that applications can choose to ignore (or include)
  regions where movability flags prove unsamplability.
Input search also offers control over the sampling distribution
  and tunes the subdivision to the target distribution
  to ensure fast convergence.

Put together, error intervals, movability flags, and input search
  overcome practical issues with interval arithmetic
  that plague state-of-the-art implementations
  and enable for more effective and robust error estimation.

\section{Background}

Arbitrary-precision interval arithmetic with recomputation
  is the standard method for approximating
  the exact result of a mathematical expression.

\subsection{Arbitrary-precision Floating Point}

Software libraries like MPFR~\cite{mpfr} 
  provide floating-point arithmetic
  at an arbitrary, user-configurable precision $p$.
This arithmetic represents a finite subset
  $S_p$ of the extended reals $\mathbb{R} \cup \{ \pm \infty \}$.
Any real number $x$ can be \textit{rounded} $R_p(x)$
  to precision $p$ by monotonically choosing
  one of the two values in $S_p$ closest to $x$;
  we write $R_p^\downarrow$ and $R_p^\uparrow$
  for rounding down and up specifically.
Precision $p$ implementations
  of functions like $\sin$ or $\log$
  are \textit{correctly rounded}:
  the implementation $f_p$ of a function $f$
  satisfies $f_p(x) = R_p(f(x))$.
But larger expressions can be inaccurate
  even in high precision:
  $(1 + x) - x$ evaluates to $0$ for all inputs $x > 2^p$.

\subsection{Interval Arithemtic}

Intervals are an abstract domain over the extended reals:
  the interval $[a, b] \in S_p \times S_p$
  represents the set $\{ x \in \mathbb{R}^* \mid a \le x \le b \}$.
Interval versions $f_p$ can be constructed
  for each mathematical function $f$
  to guarantee both soundness,
\[
\forall x_1 \in I_1, x_2 \in I_2, \ldots, \quad
f(x_1, x_2, \ldots) \in f_p(I_1, I_2, \ldots),
\]
and weak completeness,
\[
\exists
  \left. x_1, y_1\in I_1\right.,
  \left. x_2, y_2 \in I_2 \right., \dotsc,
  \quad
  f_p(I_1, \ldots, I_n) =
  [
    R_S^\downarrow(f(x_1, \ldots, x_n)),
    R_S^\uparrow(f(y_1, \ldots, y_n))
  ].
\]
The points $x_i$ and $y_i$ are called \textit{witness points},
For example, since $[a_1, a_2] + [b_1, b_2]$
  is at least $a_1 + b_1$ and at most $a_2 + b_2$,
  the result must be
  $[R_p^{\downarrow}(a_1 + b_1), R_p^{\uparrow}(a_2 + b_2)]$.

\subsection{Recomputation}

Consider an expression $e$ over the real numbers,
  consisting of constants, variables, and function applications.
In a \textit{target precision} $p$,
  the most accurate representation of its value at $x$
  is the \textit{ground-truth} value $y^* = R_p(\ll e \rr(x))$,
  where $\ll e \rr(x)$ is the error-free real-number result.
Interval arithmetic can be used to compute $y^*$.
Evaluate $e$ to some interval $[y_1, y_2]$ at precision $q > p$;
  then
\[
y_1 \le \ll e \rr(x) \le y_2
\implies
R_p(y_1) \le y^* \le R_p(y_2)
\]
by soundness and monotonicity.
If $[y_1, y_2]$ is sufficiently narrow,
  $R_p(y_1) = y^* = R_p(y_2)$,
  which computes $y^*$.
We call these narrow intervals \textit{one-value} intervals;
  computing one just requires finding a large enough precision $q$.

A useful strategy for finding $q$ is to start just a bit past $p$
  and then grow $q$ exponentially,
  which guarantees that large precisions are reached quickly.
Under certain assumptions, this exponential growth strategy
  can also be proved optimal in the sense of competitive analysis.
Sometimes, though, the $q$ required is too large to be feasible
  or may not exist due to limitations like a maximum exponent size.
Practical implementations must thus cap the maximum value of $q$
  and return indeterminate results when the cap is reached.

\section{Boolean and Error Intervals}
\label{sec:impl}

Interval arithmetic is a method for establishing bounds
  on the output of a mathematical expression.
This section introduces \textit{boolean intervals},
  which incorporate control flow into the interval arithmetic framework,
  and demonstrate how their application to domain errors
  (\textit{error intervals})
  allows integrating rich notions of input validity
  with the soundness and recomputation
  provided by interval arithmetic.

\subsection{Modeling Control Flow}

Besides real computations,
  floating-point expressions contain
  explicit and implicit control flow.
However, since intervals represent a range of inputs,
  the results of a boolean expression can be indeterminate
  (true on some inputs in the range and false on others).
Boolean intervals (for explicit control flow)
  and error intervals (for implicit control flow)
  represent this possibility.

\begin{definition}
\textit{Boolean interval}s $[\top, \top]$, $[\bot, \bot]$, and $[\bot, \top]$
  denote the sets $\{\top\}$, $\{\bot\}$, and $\{\bot, \top\}$ of booleans.
\end{definition}
\noindent
Boolean intervals form a standard three-valued logic,
  with $[\bot, \top]$ representing the indeterminate truth value
  and with standard semantics for conjunction, negation, and disjunction.
Comparison operators on intervals can return boolean intervals:
  $[1, 3] < [4, 5] = [\top, \top]$, while $[1, 4] < [3, 5] = [\bot, \top]$.
Explicit control flow like the $\F{if}$ operator
  unions its two branches when the condition is indeterminate.

Expressions also have implicit control flow due to domain errors.
As with comparisons, whether a domain error occurs
  can be indeterminate;
  each expression thus returns an \textit{error interval},%
  \footnote{Error intervals can be viewed as an abstraction
    not just of the real results of the computation,
    but of the \textsf{Error} monad the computation occurs within.}
  a boolean interval describing
  whether a domain error occurred
  during the evaluation of the expression;
  we write $\F{err}(e)$ for the error interval of expression $e$.%
  \footnote{In this way, \F{err} is analogous
    to a try/catch/else block.}
When an error is possible but not guaranteed,
  the function's output is still meaningful.
For example, the expression $\F{pow}(x, y)$
  raises a domain error when $x$ is negative and $y$ is non-integral.
Thus, $\F{pow}([-1, 2], [1, 5])$
  returns $[-1, 32]$ with error interval $[\bot, \top]$
  because $\F{pow}(-1, 1) = -1$ is the minimum possible output,
  $\F{pow}(2, 5) = 32$ is the maximum possible output,
  and $\F{pow}(-1, 2.5)$ demonstrates that domain error is possible.%
\footnote{When an expression's error interval is $[\top, \top]$, of course,
      the interval bounds are meaningless by necessity.}

At higher precisions input intervals are narrower,
  so errors that were once possible may be ruled out.
Error intervals make it possible to distinguish those cases
  (which have error interval $[\bot, \top]$)
  from cases like $\sqrt{-1}$
  (which have error interval $[\top, \top]$)
  where an error is guaranteed even at higher precision.
Consider $\sqrt{\cos(x)}$.
At low precisions and with large $x$ values,
  $\cos(x)$ can evaluate to $[-1, 1]$,
  so that $\sqrt{\cos(x)}$ is possibly invalid.
Without error intervals,
  $\sqrt{\cos(x)}$ would indicate the possible error
  by returning $[\F{NaN}, 1]$.
The error interval $[\bot, \top]$ gives more information:
  it indicates that the expression could be a valid at a higher precision.
But if, at a higher precision, $\cos(x)$ is found to be strictly negative,
  the error interval will be $[\top, \top]$ and cut off further recomputation.

\subsection{Point Validity}

Both explicit and implicit control flow
  can affect whether a point is a valid input.
Combinations of boolean and error intervals
  allow expressing these rich notions of input validity.
Consider how four validity requirements,
  drawn from the \herbie tool~\cite{herbie},
  can be formalized as expressions yielding boolean intervals.
Each validity requirement applies
  to the input $(x, y, \dotsc)$ of an expression $E$.

First, Herbie requires each input variable
  to have a finite value in the target precision.
Given the target precision's
  smallest and largest finite values $K_-$ and $K_+$,
  the comparison $K_- \le x \le K_+$
  expresses that requirement for the variable $x$.
The constants $K_-$ and $K_+$ can be computed
  by ordering all values in the target precision
  and taking those just after/before $\mp\infty$.
Note that the expression $-\infty < x < +\infty$
  is equivalent in the target precision,
  but is much weaker in higher precisions
  that contain values between $K_+$ and $+\infty$.

Second, Herbie requires the output value,
  when computed without rounding error, to be finite as well.
The expression $K_- \le E \le K_+$ enforces this requirement.
Note that boolean intervals allow incorporating
  arbitrary mathematical formulas into validity requirements.

Third, Herbie requires $E$ to not raise a domain error.
The expression $\neg \F{err}(E)$,
  which uses $\F{err}$ to access the error interval,
  enforces this requirement.
Since interval operations are meaningful
  even when an input interval contains invalid points,
  combining this requirement with the second one
  allows ignoring points that are \textit{either}
  too large or invalid without using higher precision
  to determine which restriction applies.

Fourth, Herbie users can provide a precondition
  that valid inputs have to satisfy.
That precondition can be interpreted
  as a boolean interval expression
  to enforce this requirement.
Since interval arithmetic already accounts for rounding error,
  the precondition is checked without rounding error.
If the chosen precision isn't high enough,
  the precondition returns the indeterminate boolean interval $[\bot, \top]$,
  and it is automatically recomputed at a higher precision
  to determine whether the input is in fact valid.

Since boolean intervals admit conjunction,
  all four of these requirements can be combined
  into a single formula that for input validity.
That single formula is crucial to input search \Cref{sec:search}.

\section{Movable and Immovable Intervals}
\label{sec:movable}
  
For some inputs,
  computing a ground-truth value
  using interval arithmetic and recomputation
  is impossible;
  one common problem is an overflow
  that persists across precisions.
When no ground-truth can be computed for certain input points,
  the user must be warned,
  and recomputation must be prevented from causing timeouts.
\textit{Movability flags} on each endpoint achieve this goal.
Movability flags are set when operations overflow,
  and track the influence of that overflow on further computation.
When overflow prevents the computation of a ground-truth value,
  movability flags warn the user without timing out.

\subsection{Movable and immovable intervals}

Formally, an interval is augmented with two movability flags,
  one associated with each of its endpoints;
  an endpoint is \textit{immovable} when its movability flag is set.
In examples, we write exclamation marks for immovable endpoints,
  so that $[0, 1 !]$ is an interval with a movable left endpoint $0$
  and an immovable right endpoint $1$.
The movability flags describe how intervals get narrower
  at higher precision.

\begin{definition}
  One interval refines another, written $[a', b'] \prec [a, b]$,
  when $b' = b$ and both are immovable
  or when $b' \le b$ and $b$ is movable;
  and likewise $a' = a$ and both are immovable or when $a' \ge a$
  and $a$ is movable.
\end{definition}
The goal of movability flags is to 
  determine whether recomputation at higher precision
  could possibly yield a narrower interval output.
Specifically, they guarantee:
\begin{theorem}
  \label{thm:movability}
  $\ll e \rr_q(x) \prec \ll e \rr_p(x)$
  for all $e$, $x$, and $q > p$.
\end{theorem}
\noindent
When an expression evaluates to an interval
  that is not  one-value and which has two immovable endpoints,
  any higher-precision computation will do the same,
  and recomputation is futile.
Alternatively, an interval with a single immovable endpoint
  is guaranteed to produce that value (if any).

Program inputs and constants (named and numeric)
  form a particularly simple case of \Cref{thm:movability}.
Program inputs are drawn from the target precision and
  can be directly represented in higher precision;
  they are thus represented by intervals $[! x, x !]$
  with both endpoints marked immovable.
Constants like $\pi$, however,
  cannot be exactly represented
  and have both endpoints of their interval marked movable.

For expressions larger than just constants,
  \Cref{thm:movability} is proved by induction.
Operators in the tree must properly
  interpret movability flags on their inputs
  and set movability flags on their outputs.
The key is to ensure that each operator $f$
  preserves refinement for its inputs:
\begin{property}
  \label{prop:ifun}
  $f_q(I_1', \dotsc, I_n') \prec f_p(I_1, \dotsc, I_n)$
  when $q > p$ and all $I_i' \prec I_i$.
\end{property}
\noindent
Importantly, if $f_p(I, J)$ has an immovable endpoint,
  then $f_q(I', J')$ must share that immovable endpoint.
When all operators satisfy \Cref{prop:ifun},
  \Cref{thm:movability} is preserved by structural induction.
The rest of this section describes how to guarantee \Cref{prop:ifun}
  for various operations.

\subsection{Immovability from Overflow}
  
Movability flags are set
  when the limitations of arbitrary-precision computation
  make it impossible for higher precision
  to make the result of a computation more precise.

Consider the $\exp(x)$ function, which computes $e^x$.
The MPFR library represents arbitrary-precision values
  as $\pm (1 + s) 2^h$, where the exponent $h$
  has a fixed upper bound $H$.
It thus cannot represent any values larger than $2^{H+1}$,
  so inputs $x > \log(2^{H+1})$ to $\exp$
  overflow at any precision.%
\footnote{The value of $H$ for MPFR depends on the platform,
    so \name calculates its overflow threshold at runtime,
    using 80 bits of precision and consistent rounding up.}
This fact allows $\exp$ to set movability flags on its output.

Consider the computation $\exp_p([0, 10^{10} !]) = [1, \infty !]$.
The output's right endpoint is $\exp(10^{10})$,
  which rounds up to $+\infty$.
Since $10^{10}$ exceeds the overflow threshold,
   the rounding will occur at any precision $p$;
  and since $10^{10}$ is immovable in the input interval,
  it will be the right endpoint of any refinement of the input.
So it is valid to mark the right output endpoint immovable.
Alternatively, for inputs like $[10^{10}, 10^{11}]$
  whose left endpoint exceeds the overflow threshold,
  the output's right endpoint can be marked immovable
  even if both input endpoints are movable.
Thus outputs can have immovable endpoints
  even with entirely movable inputs.
Note, however, that in this case
  the left output endpoint cannot be marked immovable,
  since it will be the largest representable finite value
  and that can typically increase with higher precision.

This general logic of overflow extends
  to the $\F{exp2}$ function with threshold $H+1$,
  and to $\F{pow}$, which detects overflow
  via the identity $x^n = \exp(n \log x)$.

\subsection{Propagating Immovability}

Once one operator sets a movability flag,
  that movability flag must be propagated through later computations.
That requires keeping track of the exactness of intermediate computations.

An arbitrary-precision function $f_p$ can be thought of
  as $R_p(f(x_1, \dotsc, x_n))$,
  where $f$ is the exact, real-valued function.
But when the exact result $f(x_1, \dotsc, x_n)$
  is representable in precision $p$,
  the rounding function $R_p$ is the identity.
Since $f(x_1, \dotsc, x_n)$ is then also representable
  at all higher precisions $q$,
  we have
  $f_q(x_1, \dotsc, x_n) = f_p(x_1, \dotsc, x_n)$.
This is the key to guaranteeing \Cref{prop:ifun}.

In general, an output interval's endpoint will be immovable
  if it is the exact result of a computation on immovable inputs.
Monotonic functions are a good illustration of the logic.
Consider a single-argument, monotonic function $f$.
Its behavior on intervals is particularly simple:
\[
  f_p([a, b]) = [R_p^\downarrow(f(a)), R_p^\uparrow(f(b))].
\]
At a higher precision,
  the rounding behavior could change
  and then so would the output interval.
However, if $a$ is immovable,
  any refinement $[a', b'] \prec [a, b]$
  must have $a = a'$ and $b \le b'$.
If $f(a)$ is exactly computed,
  $R_p^\downarrow$ is the identity.
In this case:
\begin{multline*}
  f_q([a', b'])
  = f_q([a, b'])
  =  [R_q^\downarrow(f(a)), R_q^\uparrow(f(b'))] \\
  = [f(a), R_q^\uparrow(f(b'))]
  \prec [f(a), R_p^\uparrow(f(b))]
  = f_p([a, b])
\end{multline*}

In other words, for monotonic functions,
  exact computations on immovable endpoints result in immovable endpoints.
For example, $\F{sqrt}_p([0, 4 !]) = [0, 2 !]$ because
  the square root of $4$ is exactly $2$.
However, $\F{sqrt}_p([0, 2 !]) = [0, 1.414\ldots]$
  has a movable right endpoint
  because the square root of $2$ is not exact at any precision.
Arbitrary precision libraries such as MPFR
  report whether a floating-point operator is exact,
  making this rule easy to operationalize.
  
For general interval functions, the task is more complex
  because the endpoints of the output interval
  need not be computed from input endpoints.
General interval functions may be defined
  as computing $f$ on \emph{witness points} $a'$ and $b'$:
\[
f_p([a_1, b_1], \ldots, [a_n, b_n]) =
[R_p^\downarrow(f(a'_1, \ldots, a'_n)), R_p^\uparrow(f(b'_1, \ldots, b'_n))]
\]
Here, the witness points $a'$ and $b'$
  are computed as an intermediate step
  to minimize and maximize $f$ over the input intervals.
In such a scenario,
  output endpoints are immovable
  only if the witness points are guaranteed to be the same
  in all refinements of the input intervals,
  and when $f$'s output is representable:
\begin{lemma}
  \label{lemma:generalpropagate}
  Suppose the left output endpoint of $f_p$,
  computed via $f(a'_1, \ldots, a'_n)$,
  is representable in precision $p$.
  That endpoint may be marked immovable if, for all $i$, either
  1) $a'_i = a_i$ and $a_i$ is immovable;
  2) $a'_i = b_i$ and $b_i$ is immovable;
  or 3) $a_i$ and $b_i$ are both immovable and $a'_i$ is computed exactly.
  The analogous applies to $f_p$'s right endpoint.
\end{lemma}

\begin{proof}
  Since $a'_i$ minimizes $f$ over the intervals $I_0 \ldots I_n$,
    $f_p(a'_1, \ldots, a'_n)$ bounds $f$ from below.
  If, furthermore, $f_p(a'_1, \ldots, a'_n)$ is exactly computed,
    then $f_p(a'_1, \ldots, a'_n) = f_q(a'_1, \ldots, a'_n)$
    for any higher precision $q$.

  Consider conditions (1) and (2) first.
  If (1) or (2) holds, $a'_i$ is guaranteed to be a member
    of any refinement $I_i \prec [a_i, b_i]$
    by immovability.
  Since it is the witness point for $f$ at the lower precision,
    it is a valid witness point at any higher precision too.

  Alternatively, consider condition (3).
  If (3) holds,
    the only refinement of $[a_i, b_i]$ is itself,
    so $a_i'$ is always a member of that refinement.
  Since $a_i'$ is additionally computed exactly,
    it is computed identically and thus still a witness point
    at any higher precision.

  Since, under (1), (2), or (3),
    $a'_i$ is a valid witness point at any higher precision,
    the left output endpoint computed from it
    can be marked immovable.
\end{proof}

In practice, computing witness points is usually easy.
Consider the absolute value function $\F{fabs}([a, b])$:
  the left witness point is $a$, $b$, or $0$,
  and the right witness point is $a$ or $b$.
In some cases, it's best to avoid computing witness points directly.
For example, $\cos$ has a minimum of $-1$ at $\pi$.
Instead of computing $\pi$ inexactly, $-1$ can be used directly,
  so that for example $\cos([! 3, 4 !]) = [! {-1}, 0.653\dotsb]$.
In terms of \Cref{lemma:generalpropagate},
  the witness point $\pi$ is computed at infinite precision
  and then $\cos(\pi)$ is evaluated exactly.
Overflow frequently produces immovable infinite endpoints,
  but operations on infinities are generally exact,
  so \Cref{lemma:generalpropagate} retains immovability
  in these cases too.

\subsection{Function-specific Movability Reasoning}
\label{sec:special-cases}

Function-specific reasoning sometimes provides
  additional cases in which movability flags may be set.
The most common case is multi-argument functions
  where individual arguments are special values like zero or infinity.
These special values allow setting movability flags even when
  some of the input intervals are entirely movable.

Addition is an illustrative example.
Addition of anything with infinity yields infinity;
  thus, $[1, +\infty !] + [1, 2] = [2, +\infty !]$,
  with the right output endpoint immovable.
Multiplication by an immovable $0$, like addition of infinity,
  results in an immovable $0$, even if the other argument is movable.
Multiplication by infinity is similar but more complex:
  consider $[-1, 1] \times [! 1, +\infty !] = [-\infty, +\infty]$,
  where the output interval is movable
  because the movable left-hand argument may refine
  to a strictly positive or a strictly negative interval.
Thus,
  much like overflow detection requires knowing that the input interval
  is greater than a given threshold,
  handling infinite values in multiplication
  requires knowing the sign of the input interval:

\begin{lemma}
  Let $c = a'_i \times b'_i$ be an endpoint of an interval returned by multiplication.
  The endpoint may be marked immovable if:
  1) both $a'_i$ and $b'_i$ are immovable and $c$ is computed exactly; or,
  2) $a'_i$ is zero and immovable (or likewise for $b'_i$); or,
  3) $a'_i$ is infinite and immovable and $[b_1, b_2]$ does not
  contain zero (or likewise for $b'_i$).
\end{lemma}

\begin{proof}
  The first case just restates \Cref{lemma:generalpropagate}.
  In the second case, $a'_i$ is immovable, so $a'_i = 0$ is in any refinement.
  Since $0 \times b'_i = 0$ for any $b'_i$, the result is immovable.
  Finally, in the third case, while $b'_i$ may change when refined,
    its sign cannot, since $[b_1, b_2]$ does not contain zero.
  Since $a'_i$ is immovable, $a'_i \times b'_i$ takes on a fixed sign in any refinement.
  Since $a'_i$ is additionally infinity, the resulting $+\infty$ or $-\infty$ is immovable.
\end{proof}

\noindent
Immovable infinite endpoints are often caused by overflow,
  so this multiplication-specific reasoning is essential
  to detecting immovability for many expressions.

\section{Input Search}
\label{sec:search}

Most applications of interval arithmetic,
  whether using pure sampling or input generation or exhaustive testing,
  start by selecting one or many valid input points.
But for many expressions, valid inputs are rare,
  making finding one challenging.
And some applications (such as sampling-based error estimation)
  further require uniform sampling of valid inputs.
Input search addresses this issue
  by discarding invalid portions of the input space
  and focusing on valid points.

\subsection{Intervals for Uniform Sampling}

\begin{figure}
  \begin{minipage}[t]{0.53\linewidth}
  \[
  \begin{array}{l}
    T, F, O = \emptyset, \emptyset, \{ [-\infty, \infty]^n \} \\
    \K{for}\:i \in [1, \ldots, N]: \\
    \begin{block}
     O' = \emptyset \\
     \K{for}\:x \in O:\\
     \begin{block}
     y = P(x) \\
     \K{if}\:\F{stuck}(y): \; \K{warn}\:\text{``Discarding points in $x$''} \\
     \K{elif}\:y = [\top, \top]: \; T = T \cup \{ x \} \\
     \K{elif}\:y = [\bot, \bot]: \; F = F \cup \{ x \} \\
     \K{else}: \; O' = O' \cup (\F{split}\:x\:\F{along}\:(i \bmod n))
     \end{block} \\
    O = O'
    \end{block} \\
    \K{return}\:T \cup O
  \end{array}
  \]
  \end{minipage}%
  \hfill%
  \begin{minipage}[t]{0.45\linewidth}
  \caption{Pseudo-code for input search.
    Given a precondition $P$ over $n$ real variables,
      input search finds a set $T \cup O$ of disjoint hyperrectangles,
      such that $T \cup O$ contains all valid, samplable input points.
    The $T$, $F$, and $O$ sets are initialized on the first line
      and then updated over the course of $N$ search iterations.
    In each iteration, the hyperrectangles in $O$ are evaluated.
    Stuck intervals are discarded;
      intervals where the precondition has undetermined truth value
      are split into two along a dimension picked by round robin.
  }
  \label{code:search}
  \end{minipage}
\end{figure}
  
Boolean and error intervals can evaluate input validity
  for interval inputs, as shown in \Cref{sec:impl}.
They can thus prove a whole range of inputs
  to be valid or invalid at once.
Consider a validity condition $P$
  expressed via boolean and error intervals,
  and input intervals $I_1, I_2, \ldots, I_n$.
Then $P(I_1, I_2, \ldots, I_n)$
  can yield $[\top, \top]$, $[\bot, \bot]$, or $[\bot, \top]$,
  depending on whether these intervals contain
  only valid points, only invalid points, or a mix.

More formally,
  with $n$ free variables and a target precision $p$;
  the total number of input points is then $|S_p|^n$.
The \textit{hyperrectangle} $I_1 \times \dotsb \times I_n$,
  where each free variable is bound by an interval,
  has weight
\(
w = 
| S_p \cap I_1 | / |S_p| \cdot 
  \dotsb \cdot
  | S_p \cap I_n | / |S_p|
\).
Now consider a set of hyperrectangles
  that partition the total input space.
Uniformly sampling from the $[\top, \top]$ hyperrectangles
  and rejection sampling from the $[\bot, \top]$ hyperrectangles,
  both with probability proportional to $w$,
  yields valid inputs points,
  uniformly sampled from the input space.
If uniform sampling is unnecessary,
  the weights can simply be ignored;
  or if some non-uniform distribution is required
  the weight $w$ can use the cumulative distribution function
  in place of the uniform size $| S_p \cap I_k |$.

This basic idea forms the basis of a branch-and-bound algorithm
  shown in \Cref{code:search} that iteratively decomposes the
  complete input space into a subset with a higher proportion
  of valid points.
The search algorithm maintains
  three sets of hyperrectangles,
  $F$, $T$, and $O$,
  initialized with $T = F = \{\}$
  and with $O = \{ S_p^n \}$ containing the whole input space.
At each step, the validity expression
  is evaluated on each hyperrectangle in $O$.
If the result is $[\bot, \bot]$ or $[\top, \top]$,
  the interval is moved to $F$ or $T$,
  but in the $[\bot, \top]$ case
  the interval is subdivided and both halves are placed back into $O$.
At each step, the intervals in $F \cup T \cup O$ cover the input space;
  the process terminates once no intervals are left in $O$,
  or after a fixed number of iterations (\SearchIterations in our implementation).
Thus, as the search algorithm proceeds,
  $T$ and $F$ accumulate hyperrectangles
  covering valid and invalid regions of the input space,
  while $O$ contains increasingly small hyperrectangles
  whose validity is unknown.
Once search is finished,
  points are sampled proportionally to weight
  from the hyperrectangles in $T \cup O$---%
  from $T$ directly and from $O$ via rejection sampling.

Two special twists are required
  to make this algorithm effective and maximally general.
First, when splitting $[\bot, \top]$ hyperrectangles,
  the algorithm performs best when
  the two resulting hyperrectangles contain
  the same number of values.
Since floating-point binary representations are sorted,
  the best split is midway between
  the \textit{binary representations} of the endpoints
  in the target precision.
Second, it is important that the hyperrectangles in $T \cup F \cup O$
  strictly partition the input space---%
  no point is in two hyperrectangles.
However, recall that intervals are inclusive on both ends,
  so one point can end up in two hyperrectangles
  and end up oversampled.
The fix for this is to increment one endpoint to the next
  floating-point value when splitting a hyperrectangle.
The split value is represented in the
  resulting lower hyperrectangle, and the next value is represented in the
  resulting higher hyperrectangle.
This allows input search to guarantees uniform sampling
  when that is required, such as for sampling-based error estimation.

\subsection {Range Analysis}

Hyperrectangle subdivision aids sampling
  because partitioning the input space can model
  complex relations between variables.
However, it is inefficient for the most common type of condition:
  constant bounds on variables.
Input search uses a static analysis of the precondition,
  called range analysis, to address this issue.
The main role of range analysis
  is to detect comparisons between variables and constants,
  and propagate these comparisons through boolean operations.
For each variable, the result is a set of intervals,
  where inputs outside the set are provably invalid.
Input search then forms the cartesian product
  of these per-variable interval sets
  and uses the resulting hyperrectangles
  as a starting point for the subdivision search.

Formally, range analysis returns disjoint intervals $X_{i1} \cup X_{i2} \cup \ldots \cup X_{in}$
 for each variable $x_i$, where:
\[
\ll P \rr(x) \implies \forall i, x_i \in X_{i1} \cup X_{i2} \cup \ldots \cup X_{in}.
\]
Range analysis traverses the precondition from the bottom up.
A comparison $x < C$ becomes the range table $x \mapsto \{ [-\infty, C] \}$,
  and boolean operations correspond
  to straightforward intersections and unions of range tables.
For non-trivial comparisons like $\exp(x) < y$,
  range analysis just returns a non-restrictive range table;
  these more complex cases
  are better handled by branch-and-bound search.

\subsection{Unsamplable Inputs}

As discussed in \Cref{sec:movable},
  some valid input points are \textit{unsamplable}:
  interval arithmetic cannot compute a ground truth value
  and so cannot accurately measure error.
Input search must warn the user if it finds unsamplable points,
  ideally providing an example to help the user debug the issue.
But unsamplable points often represent unintended inputs;
  users often react to the warning by adding a precondition
  ruling the unsamplable points invalid.
So many applications of interval arithmetic
  in fact discard unsamplable points, after issuing a warning.
Since unsamplable inputs often come in large contiguous regions,
  input search ought to detect them, warn the user, and ignore them.

Movability flags let input search do this.
Define an interval to be \F{stuck}
  when it has two immovable endpoints but is not one-value.%
\footnote{If only finite values are valid,
  intervals with a single immovable, infinite endpoint
  can also be considered \F{stuck}.}
In other words, $x$ is unsamplable if $\F{stuck}(\ll e \rr_p(x))$.
Surprisingly, \F{stuck} satisfies \Cref{prop:ifun}:
\begin{theorem}
  \label{thm:hyperrect-movability}
  Suppose $\F{stuck}(\ll e \rr_p(X))$ for a hyperrectangle $X$.
  Then for any hyperrectangle $X' \prec X$ (interpreted pointwise)
    and precision $q > p$,
    $\F{stuck}(\ll e \rr_q(X'))$ also holds.
\end{theorem}

\noindent
As a corollary,
  if all intervals in a hyperrectangle $X$ have both endpoints movable,
  and $\ll e \rr_p(x)$ is \F{stuck},
  then all points $x \in X$ are unsamplable,
  because their point intervals $[x, x]$ refine $X$.

\begin{proof}
  Any refinement of a \F{stuck} interval is also \F{stuck}.
  Now, by \Cref{thm:movability},
    if $X' \prec X$,
    $\ll e \rr_q(X') \prec \ll e \rr_p(X)$.
  So if $\ll e \rr_p(X)$ is \F{stuck}, so is $\ll e \rr_q(X')$.
\end{proof}

Input search thus marks each hyperrectangle's endpoints as movable
  before evaluating the validity condition.
When a hyperrectangle evaluates to a \F{stuck} interval,
  a warning is can be raised, informing the user
  to the presence of unsamplable inputs.
To aid in debugging,
  a point is chosen from $X$
  and provided to the user
  as an example unsamplable input;
  by \Cref{thm:hyperrect-movability},
  any point in $X$ will do.
Applications can configure input search
  to then discard $X$ by moving it to the $F$ set.

\section{Evaluation}
\label{sec:eval}

We implement error intervals, movability flags, and input search
  in the \name interval arithmetic library%
\footnote{\name is written in roughly 1650 lines of Racket,
  including 650 lines of interval operators,
  650 lines of search and recomputation algorithms,
  100 lines of module headers,
  and 250 lines of tests,
  and is publicly available as free software,
  at a URL removed for anonymization.}
  and evaluate it against the state-of-the-art Mathematica 12.1.1 \mathN function.
Mathematica is a proprietary software system advertised as 
  ``the world's definitive system for modern technical computing''
  and priced at roughly \$1\thinspace500 per person per year for industrial use;
  our research is supported by an educational site license.
Mathematica is also widely used across all science and engineering domains;
  a Google Scholar search for 'Mathematica' yielded 10,000 hits,
  while an ArXiv search found 200 results for 2021 alone.
Being a proprietary system, it's impossible to know with certainty
  how \mathN works;
  however, its maximum precision flag,
  its configurability for target precision or target accuracy,
  and its overflow and underflow warnings
  all suggest that \mathN implements interval arithmetic.

Our results are plotted in \Cref{fig:mathematica}.
We find that among \TotalHardPoints challenging input points,
  \name is able to resolve \percentrivalhardbetter more
  and has \timesworsemathematicaoverallunsamplable fewer
  indeterminate results.
Its error-handling-first design (with error intervals and movability flags)
  also avoids infinite loops, memory exhaustion, and hard crashes,
  which afflict Mathematica in \overallmathematicacrashtimeoutormemory cases.
Furthermore, input search can avoid \percentpointssearchsaves of the invalid inputs
  in the first place.

\begin{figure*}
  \includegraphics[width=1.0\linewidth]{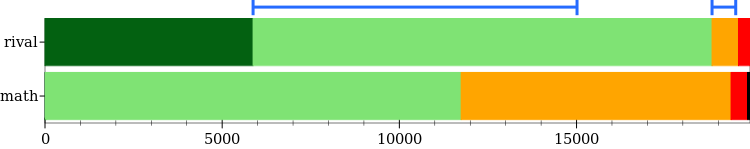}%
  \caption{
    \name versus Mathematica's \mathN
      on approximately \TotalHardPoints hard inputs
      from the \TotalHerbieBenchmarks \herbie~1.4 benchmark suite.
    \name achieves better results for each outcome threshold,
      including successfully sampled points (dark green),
      points proven invalid (light green),
      unsamplable points (orange),
      or points with unknown results (red).
    Additionally, some points cause Mathematica
      to run out of memory or crash, shown in black.
    The blue intervals mark invalid and unsamplable points
      avoided by input search.
  }
  \label{fig:mathematica}
\end{figure*}

\subsection{Methodology}

Mathematica's \mathN function can compute mathematical expressions exactly
  to an arbitrary precision, precisely the task supported by \name.
We chose Mathematica as our baseline
  after comparing multiple traditional interval arithmetic libraries.
Of all the evaluated libraries,
  Mathematica was the most robust and supported the most functions.
For example, the MPFI library
  does not handle errors soundly (unlike Rival's error intervals),
  has no built-in support for recomputation (unlike Rival's movability flags),
  and lacks support for some mathematical functions like \F{pow} and \F{erf}.
For an interval arithmetic library,
  soundness and weak completeness already require optimal interval endpoints,
  so traditional interval arithmetic libraries
  generally have the same results as \name \textit{without} its extensions.
Meanwhile, Mathematica is a widely-used commercial product
  with extensive support for all mathematical functions
  and presumably-robust implementations of all of them.
Furthermore, Mathematica's documentation suggests
  that \mathN is intended to handle invalid inputs soundly
  and cut off recomputation in some cases of over- and under-flow,
  making it at least comparable to \name.

We model our evaluation task on sampling-based error estimation;
  that is, on the task of randomly sampling inputs to an expression
  and determine the ground-truth value of the expression at those inputs.
We use a collection of \TotalHerbieBenchmarks floating-point expressions
  drawn from sources like numerical methods textbooks,
  mathematics and physics papers, and surveys of open-source code,
  collected in the \herbie~1.4 benchmark suite.
These benchmarks consist of mathematical expressions with
  \MinHerbieBenchmarkOperations--\MaxHerbieBenchmarkOperations
  operations
  and \MinHerbieBenchmarkVariables--\MaxHerbieBenchmarkVariables
  variables;
  \TotalWithConditionals benchmarks contain conditionals
  and \TotalWithPrecondition have user-defined preconditions.
Each benchmark is evaluated to double precision
  at 8\thinspace256 randomly sampled input points
  (for \mathN, to the equivalent decimal precision).
Both \mathN and \name are limited to 10240 bits;
  \name is also limited to 31-bit exponents.%
\footnote{Mathematica's precision limit is set
  via the \F{\$MaxExtraPrecision} variable%
  ~\cite{mathematica-maxextraprecision}.
  Its exponent limit is not configurable,
  but experiments suggest an exponent limit of approximately 37 bits.}
In error estimation, most inputs to an expressions
  are “easy” points where any technique will succeed;
  the crux of the problem are the challenging or marginal points.
To focus on these challenging points,
  we ignore inputs that both Mathematica and \name can sample,
  and focus on the remaining more challenging points,
  which are \TotalHardPoints out of the overall \overallallpoints.
Evaluations are performed on a machine
  with an i7-4790K CPU (at 4.00GHz) and 32GB of DDR3 memory
  running Debian 10.0 (Buster), Racket 7.9 BC, and MPFR version 4.0.2-1.

The main methodological challenge
  is matching Mathematica's and \name's semantics.
Mathematica's \mathN function supports
  floating-point, fixed-precision, and exact computation;
  we convert the sampled input points to exact rational values
  to ensure that exact computation is performed.
We also wrap each intermediate Mathematica operator
  to signal an error for \texttt{Indeterminate} or \texttt{Complex} outputs.
We compare Mathematica's and \name's outputs,
  when both tools are able to compute a value,
  to check that subtleties like \F{atan2} argument order
  are correctly aligned between the two systems.
We also capture warnings and errors,
  which Mathematica uses to indicate whether its evaluation is sound.
These cross-checks give us confidence
  that Mathematica and \name are asked
  to evaluate the same expressions on the same inputs.
(Two points fail these cross-checks at the time of this writing.
  We have investigated these inputs,
  discovered them to trigger a bug in Mathematica,
  reported the bug to Wolfram support, and had it confirmed.)
\name's input search has no direct analog in Mathematica,
  so for this experiment we randomly selected points
  without using input search.
However, we did test
  whether input search would have allowed \name
  to avoid sampling that point.

Evaluating \name against Mathematica requires understanding
  Mathematica's behavior on invalid and challenging inputs.
Unfortunately, unlike \name,
  Mathematica is proprietary, and the documentation is not always specific.
Like \name, Mathematica's \mathN attempts to detect
  invalid inputs and futile recomputations.
In general,
  for invalid inputs Mathematica either returns a complex number
  or the special \texttt{Indeterminate} value
  for some intermediate operator.
In \Cref{fig:mathematica},
  these cases are considered proven domain error (light green).
Mathematica also raises a variety of warnings,
  including for overflow and underflow;
  we conservatively assume these warnings are sound
  and are reached at low precision,
  and mark these cases unsamplable (orange in \Cref{fig:mathematica}).
For \name this status is used only when the interval is proven \F{stuck}
  before the maximum precision is reached.
(More broadly, we interpret ambiguous cases as generously as possible
  toward Mathematica to ensure a fair comparison.)
When \name or Mathematica reach their internal precision limit,
  we use the unknown result status (red in \Cref{fig:mathematica}).

\subsection{Results}

Of the \TotalHardPoints hard input points
  across our \TotalHerbieBenchmarks benchmarks,
  Mathematica is able to resolve only \mathematicahardsamplesorerror,
  while \name is able to resolve \rivalhardsamplesorerror points (\percentrivalhardbetter better).
Specifically,
  \name samples \rivalhardsamples points and proves a further \rivalharderror to be invalid;
  Mathematica samples only \mathematicahardsamples and proves only \mathematicaharderror invalid.
\name's error intervals,
  which integrate the detection of invalid inputs with recomputation,
  account for the difference in the number of points proven invalid.

Not only does \name resolve more cases,
  it resolves nearly a superset of the cases resolved by Mathematica.
\name resolves \totalrivalsamplesorerror points not resolved by Mathematica;
  Mathematica resolves only \totalmathematicasamplesorerror points that \name cannot.
(Without access to Mathematica's internals it's hard to say
  how Mathematica resolves those \totalmathematicasamplesorerror points.
Our best guess is that it may be simplifying the benchmark expressions
  before evaluating them at an input.)
Furthermore,
  when Mathematica is unable to resolve an input,
  it sometimes (\overallmathematicacrashtimeoutormemory cases)
  enters an infinite loop
  or declares that insufficient memory is available.%
  \footnote{We limit the computation to one second per input point;
    all 32GB of memory are available to Mathematica.
    \name never runs out of memory,
      or takes longer than $250\,\text{ms}$ to complete a computation.}
We render these inputs in black in \Cref{fig:mathematica};
  \name never crashes in this way.
Moreover, in \overallmathematicacrash cases, the Mathematica process crashes
  and must be killed with \F{SIGKILL},
  which may cause the user to lose work.

Both \name and Mathematica are unable to resolve some inputs:
  \rivalhardunresolved points for \name and \mathematicahardunresolved for Mathematica.
In these cases it's important to give additional information to the user:
  some of these points may be resolvable with more precision, time, or memory,
  but others cannot be resolved due to algorithmic limitations.
\name's movability flags, which prove \overallrivalunsamplable points unsamplable,
  soundly detect algorithmic constraints;
  Mathematica's various warnings affect
  \overallmathematicaunsamplable points (\timesworsemathematicaoverallunsamplable worse)
  and are unsound, sometimes triggered even for samplable inputs.
Mathematica marks unsamplable \totalrivalsamplablemathematicaunsamplable points
  that \name successfully samples as a result.
Thus, unlike \name's movability flags,
  they do not advise whether the user
  ought to try again with more precision
  or abandon the computation.
At the same time,
  Mathematica reaches its internal precision limit \overallmathematicatimesrivalunknown more often
  than \name because \name's movability flags allow it
  to identify unsamplable points at low precision.

For applications that need to sample valid inputs,
  \name's input search is additionally helpful.
Input search is able to avoid sampling
  \pointssearchsavesfromrivaldomainerror invalid inputs
  and \pointssearchsavesfromrivalunsamplable unsamplable inputs,
  thereby cutting the rate of invalid and unsamplable points
  by \percentpointssearchsaves
  (the blue intervals in \Cref{fig:mathematica}).
This means that using \name for sampling-based error estimation
  would yield even better results
  than those suggested by \Cref{fig:mathematica},
  since input search allows \name to avoid many challenging, invalid points.
The overall picture of these results is that \name is more capable than Mathematica:
  it is able to sample more points, discard more invalid points,
  and leave fewer indeterminate cases unresolved.
It also does so without timeouts, memory exhaustion, or fatal crashes.
Yet it also offers sound guarantees and valuable extensions like input search.

As a result of \name's improvements in interval arithmetic,
  \name is \timesrivalfastermathematica faster than Mathematica.
Mathematica’s timeouts, OOM, and crashes also add to its runtime,
  but \name is \timesrivalfastermathematicawithoutcrashestimeouts
  faster than Mathematica even after excluding these inputs.
This is a consequence of \name's better handling
  of invalid inputs and futile recomputation:
  any interval arithmetic library spends the bulk of its time
  doing high-precision arithmetic.
For one library to be faster than another, then,
  that library must identify
  invalid inputs, detect futile recomputation,
  sample fewer such points, or in some other way
  do fewer high-precision operations.
That’s precisely what \name’s extensions do;
  those extensions ultimately require just a few comparisons,
  so they add negligible overhead time,
  but they dramatically cut down on unnecessary computation.
In fact, \name is also \percentfasterthanmpfi faster
  than MPFI, another commonly-used interval arithmetic library,
  again due to \name's ability to skip unnecessary computation.

\section{Detailed Analysis}
\label{sec:analysis}

This section expands on \Cref{sec:eval}
  to detail the performance
  of error intervals, movability flags, and input search
  individually,
  and to describe particularly challenging benchmarks.
The data show that \name's extensions are effective,
  accomplish their task at low precision,
  and handle challenging inputs from a range of domains.

\begin{figure*}
  \begin{minipage}[t]{.32\linewidth}
    \includegraphics[width=1.0\linewidth]{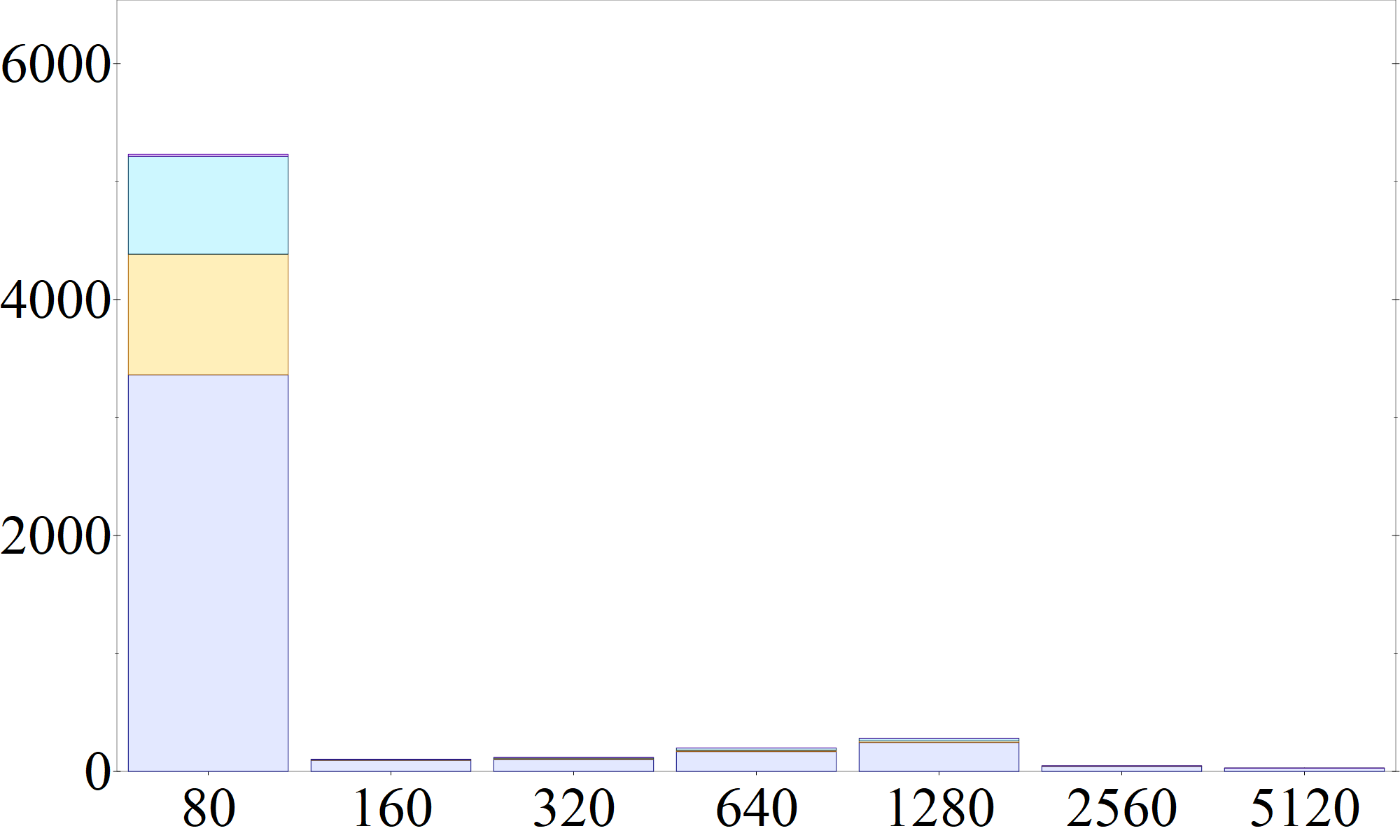}%
    \caption{
      Counts of valid points (bottom, blue),
        domain errors (middle, yellow),
        infinite outputs (top, light blue),
        and precondition failures (too rare to be visible)
        for points sampled without input search.
      80 bits of precision suffices for most points.
    }
    \label{fig:point-precision}
  \end{minipage}\hfill%
  \begin{minipage}[t]{0.32\linewidth}
    \includegraphics[width=\linewidth]{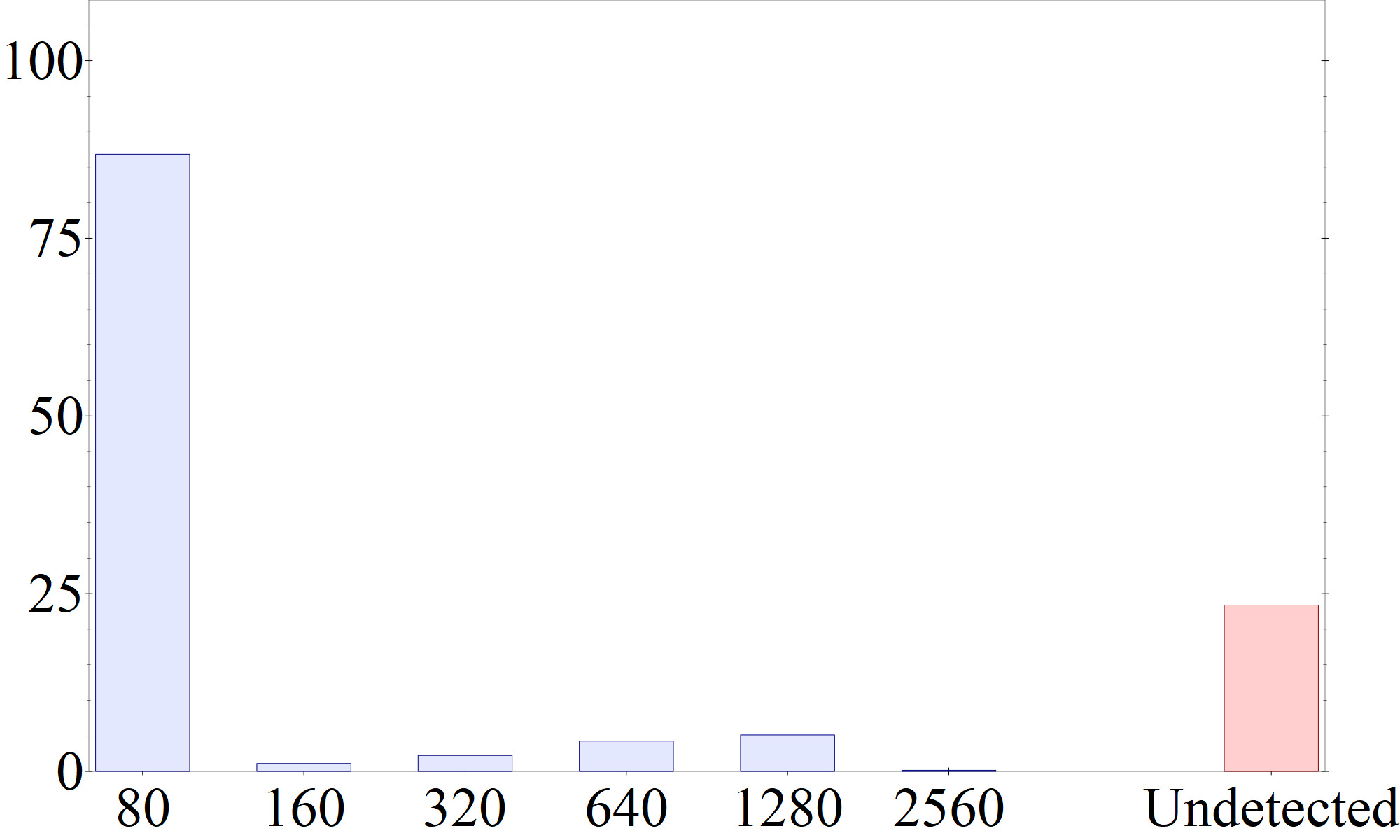}
    \caption{
      Counts of unsamplable points.
      \name's movability flags detect all but
        \PercentUnsamplablePointsUndetectedSearchDisabled
        of unsamplable points,
        with the majority detected at just 80 bits of precision.
      Detecting unsamplable points avoids timeouts and user disappointment.
    }\label{fig:movability}
  \end{minipage}\hfill%
  \begin{minipage}[t]{0.32\linewidth}
    \includegraphics[width=\linewidth]{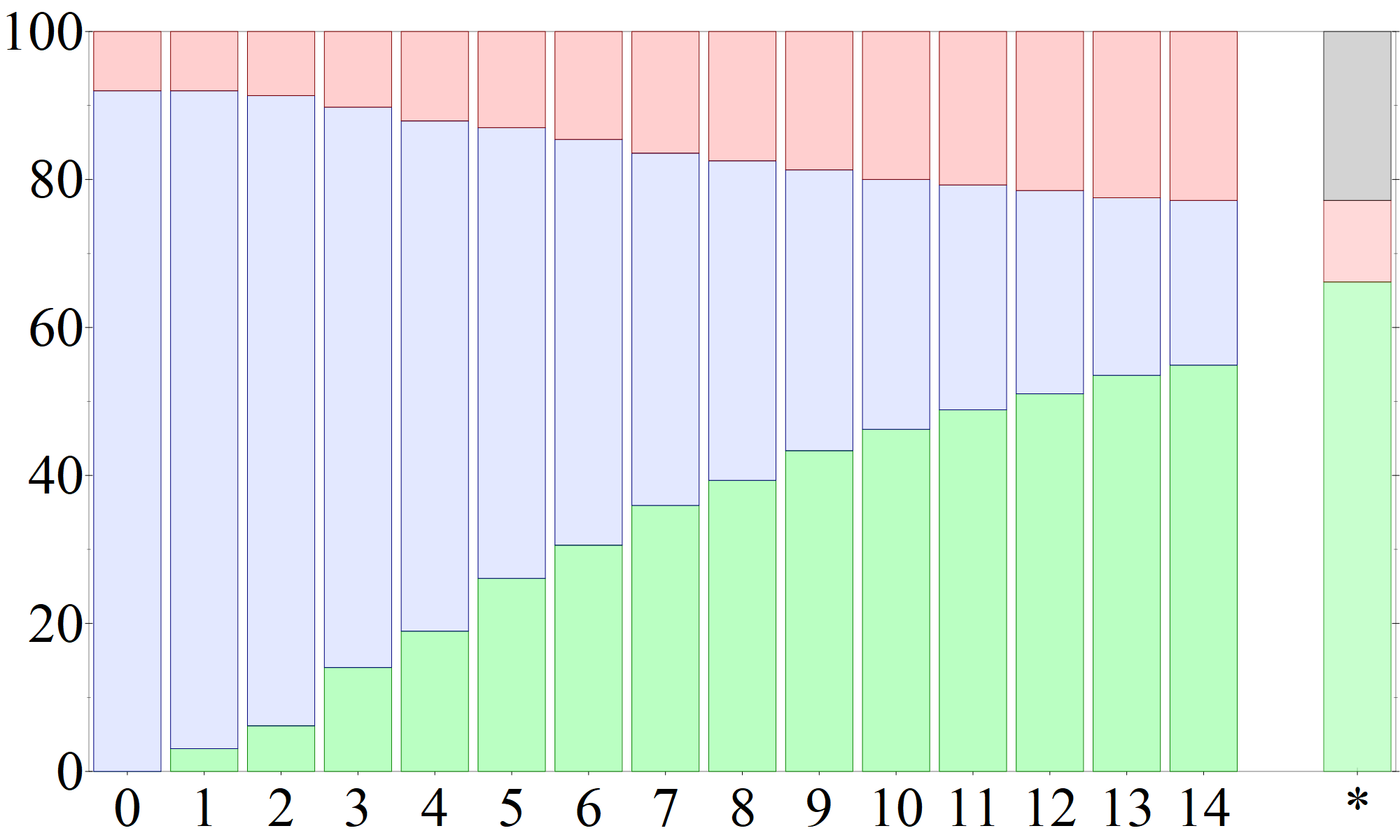}
    \caption{
      Size of the
        $T$ (bottom, green), $O$ (middle, blue), and $F$ (top, red) sets
        after each iteration of \name's input search.
      The ``0'' bar shows the results of range analysis;
        the ``$*$'' bar shows the results of rejection sampling.
      On average,
        input search avoids \SearchInvalidSampleReduction of invalid samples.
    }\label{fig:input-search}
  \end{minipage}%
\end{figure*}

\subsection{Error Intervals}
\label{sec:error-eval}

Error intervals are evaluated by considering
  the internal precision at which they deem input points valid or invalid.
Input search is disabled.
\Cref{fig:point-precision} counts valid and invalid points
  by the internal precision required.
The vast majority of points are proven valid or invalid
  at 80 bits of precision,
  with 1280 bits the second-most common precision.
This curious runner-up precision
  is largely due to trigonometric functions:
  the largest double-precision values are as large as $2^{1024}$,
  and 1280 bits is enough precision
  to identify their period modulo $\pi$.
Recomputation automatically increases the precision to this level.

Compared to valid points,
  a larger fraction of invalid points are detected at 80 bits of precision,
  likely because for these points,
  \name doesn't need to compute a ground truth value,
  just prove that the point is invalid.
The majority of invalid points
  result from infinite outputs and domain errors,
  while user-specified precondition failures are rare.%
\footnote{Infinite inputs are totally absent,
  unsurprising given the small fraction of inputs they represent.}
Specifically among expressions with preconditions,
  \PercentFailPreconditionOfBenchmarksWithPrecondition
  of sampled points are invalid,
  suggesting that error intervals work well
  even under these most challenging conditions.

The benchmark named ``Harley's example''%
\footnote{This benchmark appears as a case study
  in the original Herbie paper, suggesting its importance~\cite{herbie}.}
  demonstrates the importance of soundness
  for error and boolean intervals:
\[
\begin{array}{l}
  \K{pre}\:c_p > 0 \land c_n > 0 \: \\
  \K{let}\:a = 1 / (1 + e^{-s}), b = 1 / (1+e^{-t})\:\K{in} \\
  \left( a^{c_p} \cdot (1 - a)^{c_p}\right) / \left(b^{c_n} \cdot (1 - b)^{c_n}\right)
\end{array}
\]
Some inputs to this expression produce outputs just below
  or just above the point where the output rounds to infinity;
  error intervals force recomputation
  until it is clear whether which of the points are valid,
  which can require as many as 5120 bits.

\subsection{Movability Flags}
\label{sec:movable-eval}

Movability flags are evaluated by considering
  the precision at which they can prove a point unsamplable.
\Cref{fig:movability} shows that movability flags
  detect a \PercentUnsamplableDetected of unsamplable points
  and thus prevent most futile recomputations.
Unsamplable points are found (and warnings issued)
  for \NumBenchmarksWithUnsamplablePointsSearchDisabled
  input expressions,
  mostly involving ratios of exponential functions.
For example, the \F{expq2} benchmark
  involves the expression $\exp(x) / (\exp(x) - 1)$;
  its unsamplable points can be detected with just 80 bits of precision.
For more difficult expressions, higher precision is sometimes necessary
  to prove unsamplability;
  consider the \F{expq3} benchmark:
\[
\frac{\varepsilon \cdot \left(e^{(a + b) \cdot \varepsilon} - 1\right)}
{\left(e^{a \cdot \varepsilon} - 1\right) \cdot \left(e^{b \cdot
    \varepsilon} - 1\right)},
\text{ with } {-1} < \varepsilon < 1.
\]
\noindent
This benchmark is unsamplable when
  $a = \varepsilon = 10^{-200}$ and $b = 10^{200}$:
  the second term in the denominator overflows,
  but at low precisions the first rounds down to $0$,
  making the denominator possibly zero.
Only at 2560 bits of precision%
\footnote{Because, with a minimum double-precision exponent of $-1024$,
  the term $e^{a \cdot \varepsilon}$
  needs at least $2048$ bits to clearly differ from $1$.
}
  does it become clear that the denominator is very large,
  which is necessary to prove the expression unsamplable.

By contrast, cases where movability flags fail
  often seem like oversights by the benchmark developers
  or the users those benchmarks are derived from.
For example, the \F{regression} suite's \F{exp-w} benchmark,
  $e^{-w} \cdot \F{pow}(\ell, e^w),$
  fails to detect unsamplable points
  when $\ell$ is negative and $w$ large and positive,
  $\F{pow}(\ell, e^w)$ raises a negative number
  to a possibly-infinite power.
Since \name does not track movability for error intervals,
  it does not detect that recomputation here is futile.
But the author of this expression
  likely intended $\ell$ to be close to $1$,
  instead of negative.
Similarly, movability flags fail
  on textbook examples of difficult-to-analyze functions
  like ``Kahan's Monster'':
\[
\begin{array}{l}
  \K{pre}\:y > 0 \\
  \K{let}\:z =
    \left(|y - \sqrt{y^2 + 1}| - 1/(y + \sqrt{y^2 + 1})\right),
    z_2 = z \cdot z\:\K{in}\\
  \K{if}\:z_2 = 0\:\K{then}\:1\:\K{else}\:(\exp{z_2} - 1) / z_2
\end{array}
\]
In the reals, $z$ is exactly $0$ so the program yields $1$;
  with interval arithmetic, $z$ is a narrow interval straddling $0$
  and \name cannot rule out a division by zero in the \F{else} branch.
Of course, since real computation is undecidable,
  constructing counterexamples like this will be possible for any algorithm.

\subsection{Input Search}
\label{sec:search-eval}

Input search is evaluated by considering
  the fraction of the input space in the $T$, $O$, and $F$ sets.
\Cref{fig:input-search} shows these fractions
  as a percentage of the total search space,
  averaged over all benchmarks,
  at each input search iteration.
Over \NumberOfSearchIterations iterations,
  input search discards \PercentSpaceF of the total input space
  and proves an additional \PercentSpaceT to contain only valid points.
As a result, \GuaranteeSampleChance of sampled points
  are guaranteed to be valid and do not need to be rejection-sampled,
  and \percentpointssearchsaves of invalid points are avoided.

Input search helps \name sample from benchmarks
  that would otherwise be infeasible to sample from.
Consider \F{Jmat.Real.erfi}:%
\footnote{Specifically, of the two benchmarks with that name,
  the one labeled ``branch $x$ greater than or equal to 5'',
  which we assume is mislabeled.}
\[
  \left( \frac1{\sqrt\pi} e^{x \cdot x} \right) \cdot \left(\frac1{|x|} +
  \frac12 \frac1{|x|^3} + \frac34 \frac1{|x|^5} + \frac{15}{8}
  \frac1{|x|^7} \right),
  \text{ with } x \ge \frac12
\]
Inputs above approximately $30$ overflow,
  so roughly 0.15\% of double-precision floating-point values are valid.
But input search identifies the valid range with ease,
  and only 1.4\% of samples are ultimately rejected.

Input search can also sometimes prove
  that there are no valid inputs at all.
In the benchmark
\[
\sqrt{\log(x)+\sin(x)} / \sin^{-1}(x+2),
\]
the term $\sin^{-1}(x+2)$ requires $-3 \le x \le -1$,
  while $\log(x)$ requires $x > 0$,
  so the benchmark has no valid points.
This benchmark was originally derived from a user submission;
  input search allows \name to provide a helpful error message
  to the (doubtless confused) user.

\subsection{Generality}

\begin{table*}
  \begin{tabular}{lr|rrrr|rr|rrr}
     & & \multicolumn{4}{c}{Error Intervals} & \multicolumn{2}{|c|}{Movability Flags} & \multicolumn{3}{c}{Input Search} \\
Benchmarks & $N$ & Valid & Pre. & Inf. & Error & Total & Detected & $|T|$ & $|O|$ & $|F|$ \\  \hline
hamming & $28$ & $76.8\%$ & $0.0\%$ & $6.3\%$ & $16.9\%$ & $4.6\%$ & $100.0\%$ & $53.9\%$ & $20.1\%$ & $26.0\%$ \\
haskell & $270$ & $67.1\%$ & $0.0\%$ & $17.3\%$ & $15.6\%$ & $0.1\%$ & $100.0\%$ & $60.5\%$ & $25.1\%$ & $14.4\%$ \\
libraries & $50$ & $69.9\%$ & $0.0\%$ & $18.2\%$ & $11.9\%$ & $3.0\%$ & $32.8\%$ & $60.1\%$ & $12.2\%$ & $27.7\%$ \\
mathematics & $39$ & $65.4\%$ & $2.8\%$ & $5.0\%$ & $26.9\%$ & $17.6\%$ & $88.1\%$ & $36.3\%$ & $20.8\%$ & $43.0\%$ \\
numerics & $38$ & $79.5\%$ & $0.0\%$ & $15.3\%$ & $5.1\%$ & $0.7\%$ & $90.6\%$ & $41.2\%$ & $8.2\%$ & $50.6\%$ \\
physics & $31$ & $51.3\%$ & $0.0\%$ & $15.9\%$ & $32.9\%$ & $6.7\%$ & $96.9\%$ & $32.9\%$ & $40.8\%$ & $26.3\%$ \\
regression & $13$ & $87.3\%$ & $0.0\%$ & $5.3\%$ & $7.4\%$ & $8.3\%$ & $33.4\%$ & $53.0\%$ & $21.0\%$ & $25.9\%$ \\
tutorial & $3$ & $90.1\%$ & $0.0\%$ & $9.9\%$ & $0.0\%$ & $0.0\%$ & $100.0\%$ & $84.8\%$ & $6.8\%$ & $8.4\%$ \\
 \hline 
Herbie v1.4 & $481$ & $67.3\%$ & $0.3\%$ & $14.8\%$ & $17.6\%$ & $3.0\%$ & $81.0\%$ & $54.9\%$ & $22.3\%$ & $22.8\%$ \\
User Inputs & $4888$ & $63.5\%$ & $0.0\%$ & $13.5\%$ & $23.0\%$ & $5.0\%$ & $55.8\%$ & $61.1\%$ & $17.6\%$ & $21.3\%$ \\
FPBench & $126$ & $79.9\%$ & $14.2\%$ & $5.9\%$ & $0.0\%$ & $2.1\%$ & $93.3\%$ & $19.8\%$ & $7.6\%$ & $72.6\%$ \\

  \end{tabular}
  \caption{
    \name's success on subsets and additional benchmark sets.
    The table shows the percentage of points
      that are valid or invalid for various reasons;
      the total fraction unsamplable points
      and the fraction of those detected by movability flags;
      and the sizes of $|T|$, $|O|$, and $|F|$ after input search.
    Results for error intervals and movability flags
      are with input search disabled
      to separate the impacts of each extension.
  }
  \label{tbl:by-suite}
\end{table*}

To demonstrate the generality of
  error intervals, movability flags, and input search
  across domains, we perform a subgroup analysis.
\Cref{tbl:by-suite} breaks down
  the Herbie~1.4 benchmarks into their 8 component suites.

\name is effective in a wide variety of circumstances.
Different benchmark suites
  range from 51.3\% (\F{physics})
  to 90.1\% (\F{tutorial}) valid points,%
\footnote{
The \F{physics} benchmarks also trip up input search
  by heavily using trigonometric functions.
The \F{tutorial} suite has three very simple expressions,
  and its invalid points all come from a single overflow.
}
  and the effectiveness of input search correspondingly varies.
Movability flags are
  \HammingPercentUnsamplableDetectedSearchDisabled
  effective in the \F{hamming} benchmark suite,
  which has many ratios of exponential functions,
  but much less effective in the \F{regression} suite,
  which is composed of difficult-to-analyze examples
  that once triggered timeouts or crashes
  in the Herbie project that compiled these benchmarks.

We also consider  two additional benchmark sets
  to evaluate \name's generality:
  the \TotalFPBenchBenchmarks standard FPBench benchmarks
  and \TotalDemoBenchmarks publically-accessible user submissions
  to the Herbie web demo.
The FPBench suite is intended
  for benchmarking floating-point analysis tools
  and makes extensive use of preconditions.
Despite this, input search samples
  a valid point with \FPBenchValidProbability probability,
  discarding \FPBenchF of the input space,
  and movability flags are extremely effective
  at detecting unsamplable points.
The user-submitted benchmarks are the other extreme:
  unorganized, duplicative, and sometimes copied from the other two suites.
\name nonetheless preforms well---especially input search,
  since there are no preconditions among the user submissions
  so detecting valid inputs is easier.
Both the subset analysis and validation sets
  thus show that error intervals, movability flags, and input search
  are effective in general across domains.

\section{Discussion}
\label{sec:discussion}

The \name library,
  containing error intervals, movability flags, and input search,
  is publicly available as open source software.%
\footnote{URL will be available after anonymization is lifted.}

\subsection{Implementing \name}

Ensuring soundness and weak completeness in \name required care.
We started by splitting each function's domain
  into \textit{monotonic regions}.
For example, \F{fmod} has monotonic regions
  defined by $k y \leq x < (k + 1) y$ for some $k$,
  monotonically increasing in $x$ and decreasing in $y$.%
\footnote{For positive $x$; \F{fmod} is odd in its first argument.}
Finding the minimum or maximum of a function
  then requires finding the maximum over all of the monotonic regions
  in the input intervals.
Soundness is guaranteed
  by the monotonicity of the function within each region,
  while weak completeness is guaranteed
  because the output interval's endpoints
  come from particular points in some monotonic region.
We wrote paper proofs for each function
  to ensure that we accurately identified the monotonic regions,
  and verified that we correctly chose
  the maximum and minimum points
  using Mathematica.

To catch one-off oversights and typos,
  we additionally deployed millions of random tests
  for soundness on both wide and narrow intervals,
  weak completeness (by subdividing intervals),
  and movability (by recomputing in higher precision).
These tests further improved \name's reliability.
One random test found that \F{sin} internally rounded $\pi$ to nearest
  instead of rounding $\pi$ differently for the two endpoints;
  this is a soundness bug.
Another random test found that
  defining $\F{fmod}(x, y)$ by $\F{trunc}(x / y) \cdot y$
  fails to account for the correlation between $x / y$ and $x$;
  this is a completeness bug
  fixed with a custom implementation of \F{fmod}.
Finally, a manual audit found that
  movability flags weren't being set when \F{pow} underflowed to $0$,
  which lead to movability flags missing some unsamplable points.
We replaced the \F{pow} movability logic
  with the identity $\F{pow}(x, y) = \F{exp}(y \cdot \F{log}(x))$
  to avoid duplicating the movability logic.

\subsection{Deploying \name}

We began developing \name three years ago
  to replace unsound sampling heuristics
  for a popular floating point analysis and error estimation tool.%
\footnote{Tool will be identified once anonymity is lifted.}
\name was deployed to production one year ago
  and has since been used by thousands of users.

\name was originally a traditional interval arithmetic library
  aimed at soundly computing error-free ground truths.
At first, our main concern was speed:
  interval arithmetic generally requires evaluating each function twice
  (once for each endpoint)
  while the existing unsound heuristics evaluated each function once,
  so we expected performance to be twice as bad.
These concerns were unfounded.
Instead, though ``difficult'' points needed higher precision in \name
  (in order to soundly prove they were correctly sampled),
  \name could sample ``easy'' points with lower precision than the earlier heuristics,
  so replacing them with \name did not slow down sampling.

Instead of speed, the core problem
  with this initial version of \name
  was its handling of invalid inputs.
Initially, \name used \F{NaN} endpoints to indicate domain errors,
  much like the open source MPFI library~\cite{mpfi}.
But those \F{NaN} endpoints would sometimes be lost in later computations,
  leading to \name erroneously computing outputs for invalid inputs.
We developed error intervals
  to guarantee that \name only sampled valid points.
With this problem fixed,
  we realized that unsamplable points
  consumed an unjustifiably large share of sampling time.
Movability flags were developed out of the realization
  that overflow was behind most of these unsamplable points
  and could be easily detected at low precision.
This sped up sampling,
  and allowed samping inputs
  for many larger expressions that used to time out.
But this only revealed
  that most of those sampled inputs were invalid or unsamplable;
  input search was developed to address this issue.

All told, \name is sounder, faster, and more user-friendly
  than the heuristic approach it replaced.
In rare cases, the earlier heuristics computed incorrect ground truth values,
  even for textbook examples such as $\sqrt{x+1} - \sqrt{x}$;
  \name never does.
Its input search and movability flags
  allow estimating error for larger benchmarks
  than the earlier heuristics could support.
And by rigorously handling error cases,
  algorithmic limitations,
  and even edge cases like benchmarks with no valid inputs,
  \name allows for clearer and more actional error messages for users.
Over all, through error intervals, movability flags, and input search,
  \name successfully addressed user-visible issues
  and improved the application it was embedded in.

\section{Related Work}

\paragraph{Computable Reals}

The problem of computing the exact value of a mathematical expression
  has a long history.
Early work~\cite{computable-reals,boehm-idea,boehm-compare,boehm-java}
  explored various representations for ``computable real'' values
  and developed algorithms for computing functions over them.
One common representation involving scaled 2-adic numbers
  yields particularly simple implementations%
  ~\cite{constructive-real-verified,constructive-real-python},
  but is slow, especially for expressions having certain forms~\cite{constructive-real-tweaks}.
Interval arithmetic was later discovered
  to be better suited to this task, as it tends to be faster~\cite{boehm-fast}.

Interval arithmetic has been widely available since 1967's XSC~\cite{xsc}.
Today,
  the Boost~\cite{boost-ivals} and Gaol~\cite{gaol} libraries
  are widely used,
  both providing interval arithmetic with double-precision endpoints.
The more modern Moore library~\cite{moore-ivals}
  provides arbitrary-precision endpoints as well.
The IEEE~1788 standard formalizes several forms of interval arithmetic~\cite{ieee-ivals},
  including interval decorations, which could be used
  to implement \name's error intervals and movability flags.
Recent work welds computable real algorithms
  to decision procedures for special cases like as rational arithmetic,
  and provides an API for supporting the combination~\cite{api-for-real-numbers}.
Despite its unique benefits for user-facing calculator applications,
  general-purpose error estimation generally lies outside the special cases,
  so \name's methods are orthogonal.
RealPaver uses interval arithmetic
  to characterize the solution sets of nonlinear constraints~\cite{applied-interval-analysis},
  to aid in pruning for global function optimization~\cite{realpaver},
  similar to \name's boolean intervals and input search.
However, since RealPaver is targetted specifically
  at mathematical optimization problems,
  it neither tracks rounding error
  nor provides error intervals
  nor guarantees uniform sampling.

\paragraph{Error Estimation}

There is a large literature on measuring floating-point rounding error.
One approach~\cite{rosa,daisy,salsa-1}
  uses one interval to track value ranges
  and another interval to track rounding errors,
  producing a sound error bound for the whole computation.
This approach can be extended
  to and intraprocedural, modular abstract interpretation~\cite{precisa,salsa-2,martel-types}.
Error Taylor forms can also bound floating-point error~\cite{fptaylor,fptuner,satire}
  and generally providing tighter bounds than interval-based bounds,
  as does semidefinite programming~\cite{real2float}.
Both error Taylor forms and semidefinite characterize error at a fixed input point
  and then apply sound global optimization, such as Gelpia~\cite{gelpia},
  which uses interval arithmetic to compute sound bounds.

An alternative approach to error estimation
  aims to generate inputs with particularly large rounding error.
S3FP~\cite{s3fp} treats the expression itself as a black box,
  while FPGen~\cite{fpgen} uses symbolic execution
  to more efficiently target problematic input ranges.
Other tools target rare errors~\cite{lsga,eagt,autornp}
  or attempt to maximize code coverage~\cite{fpse,coverme}.
In each case, these inputs are generated, selected, or mutated
  by comparing to an error-free ground-truth,
  usually computed with arbitrary precision or interval arithmetic.
\name's enhancements to interval arithmetic
  are especially attractive for input generation methodologies,
  since they often generate invalid inputs
  (which they should select against)
  or unsamplable inputs (which they should ignore).

\paragraph{Error Detection}

Some tools aim to localize rounding error within a large program.
FpDebug~\cite{fpdebug} and Herbgrind~\cite{herbgrind}
  use Valgrind to target x86 binaries
  and detect rounding error using higher-precision shadow values.
PositDebug does the same but for posit, rather than floating-point,
  values~\cite{positdebug}.
\citet{baozhang} instead detect control-flow divergence
  and re-executing the program at higher precision.
Herbie~\cite{herbie}, an error repair tool,
  uses heuristics to select a sufficiently high shadow value precision;
  the algorithm is unsound and can compute incorrect ground truths.
TAFFO replaces user-selected inputs with run-time values
  to dynamically substitute fixed-point for floating-point computation~\cite{taffo}.
All of these need to compute an error-free ground truth
  and are susceptible to invalid or unsamplable inputs.

\section{Conclusion}

This paper addresses three problems
  common to traditional interval arithmetic implementations.
Error intervals track the necessity and possibility of domain errors,
  and can be combined with other boolean constraints
  to provide a rich description of valid and invalid inputs.
Movability flags add another pair of boolean flags
  to track the impact of overflows,
  identifying inputs
  where algorithmic limitations
  prevent identifying a correct ground-truth value.
Finally, input search leverages
  both error intervals and movability flags
  to identify regions of valid, samplable points,
  even when those regions have low measure in the input space.
\name implements all three features and
  outperforms the state-of-the-art Mathematica 12.1.1 \mathN function,
  resolving \percentrivalhardbetter more inputs,
  returning indeterminate results \timesworsemathematicaoverallunsamplable less often,
  and avoiding all cases of fatal error.

\bibliographystyle{ACM-Reference-Format}

\begin{thebibliography}{56}


\ifx \showCODEN    \undefined \def \showCODEN     #1{\unskip}     \fi
\ifx \showDOI      \undefined \def \showDOI       #1{#1}\fi
\ifx \showISBNx    \undefined \def \showISBNx     #1{\unskip}     \fi
\ifx \showISBNxiii \undefined \def \showISBNxiii  #1{\unskip}     \fi
\ifx \showISSN     \undefined \def \showISSN      #1{\unskip}     \fi
\ifx \showLCCN     \undefined \def \showLCCN      #1{\unskip}     \fi
\ifx \shownote     \undefined \def \shownote      #1{#1}          \fi
\ifx \showarticletitle \undefined \def \showarticletitle #1{#1}   \fi
\ifx \showURL      \undefined \def \showURL       {\relax}        \fi
\providecommand\bibfield[2]{#2}
\providecommand\bibinfo[2]{#2}
\providecommand\natexlab[1]{#1}
\providecommand\showeprint[2][]{arXiv:#2}

\bibitem[\protect\citeauthoryear{??}{iee}{2015}]%
        {ieee-ivals}
 \bibinfo{year}{2015}\natexlab{}.
\newblock \showarticletitle{{IEEE Standard for Interval Arithmetic}}.
\newblock \bibinfo{journal}{\emph{IEEE Std. 1788-2015}} (\bibinfo{year}{2015}).
\newblock


\bibitem[\protect\citeauthoryear{Altman, Gill, and McDonald}{Altman
  et~al\mbox{.}}{2003}]%
        {num-issues-in-stat}
\bibfield{author}{\bibinfo{person}{Micah Altman}, \bibinfo{person}{Jeff Gill},
  {and} \bibinfo{person}{Michael~P. McDonald}.}
  \bibinfo{year}{2003}\natexlab{}.
\newblock \bibinfo{booktitle}{\emph{Numerical Issues in Statistical Computing
  for the Social Scientist}}.
\newblock \bibinfo{publisher}{Springer-Verlag}. 1--11 pages.
\newblock


\bibitem[\protect\citeauthoryear{Altman and McDonald}{Altman and
  McDonald}{2003}]%
        {num-replication}
\bibfield{author}{\bibinfo{person}{Micah Altman} {and}
  \bibinfo{person}{Michael~P. McDonald}.} \bibinfo{year}{2003}\natexlab{}.
\newblock \showarticletitle{Replication with attention to numerical accuracy}.
\newblock \bibinfo{journal}{\emph{Political Analysis}} \bibinfo{volume}{11},
  \bibinfo{number}{3} (\bibinfo{year}{2003}), \bibinfo{pages}{302--307}.
\newblock
\urldef\tempurl%
\url{http://pan.oxfordjournals.org/content/11/3/302.abstract}
\showURL{%
\tempurl}


\bibitem[\protect\citeauthoryear{{Bagnara}, {Carlier}, {Gori}, and
  {Gotlieb}}{{Bagnara} et~al\mbox{.}}{2013}]%
        {fpse}
\bibfield{author}{\bibinfo{person}{R. {Bagnara}}, \bibinfo{person}{M.
  {Carlier}}, \bibinfo{person}{R. {Gori}}, {and} \bibinfo{person}{A.
  {Gotlieb}}.} \bibinfo{year}{2013}\natexlab{}.
\newblock \showarticletitle{Symbolic Path-Oriented Test Data Generation for
  Floating-Point Programs}. In \bibinfo{booktitle}{\emph{2013 IEEE Sixth
  International Conference on Software Testing, Verification and Validation}}.
  \bibinfo{pages}{1--10}.
\newblock


\bibitem[\protect\citeauthoryear{Bao and Zhang}{Bao and Zhang}{2013}]%
        {baozhang}
\bibfield{author}{\bibinfo{person}{Tao Bao} {and} \bibinfo{person}{Xiangyu
  Zhang}.} \bibinfo{year}{2013}\natexlab{}.
\newblock \showarticletitle{On-the-fly Detection of Instability Problems in
  Floating-point Program Execution}.
\newblock \bibinfo{journal}{\emph{SIGPLAN Not.}} \bibinfo{volume}{48},
  \bibinfo{number}{10} (\bibinfo{date}{Oct.} \bibinfo{year}{2013}),
  \bibinfo{pages}{817--832}.
\newblock
\showISSN{0362-1340}
\urldef\tempurl%
\url{https://doi.org/10.1145/2544173.2509526}
\showDOI{\tempurl}


\bibitem[\protect\citeauthoryear{Benz, Hildebrandt, and Hack}{Benz
  et~al\mbox{.}}{2012}]%
        {fpdebug}
\bibfield{author}{\bibinfo{person}{Florian Benz}, \bibinfo{person}{Andreas
  Hildebrandt}, {and} \bibinfo{person}{Sebastian Hack}.}
  \bibinfo{year}{2012}\natexlab{}.
\newblock \showarticletitle{A Dynamic Program Analysis to Find Floating-point
  Accuracy Problems} \emph{(\bibinfo{series}{PLDI '12})}.
  \bibinfo{publisher}{ACM}, \bibinfo{address}{New York, NY, USA},
  \bibinfo{pages}{453--462}.
\newblock
\showISBNx{978-1-4503-1205-9}
\urldef\tempurl%
\url{http://doi.acm.org/10.1145/2254064.2254118}
\showURL{%
\tempurl}


\bibitem[\protect\citeauthoryear{Boehm}{Boehm}{1987}]%
        {boehm-idea}
\bibfield{author}{\bibinfo{person}{Hans~J. Boehm}.}
  \bibinfo{year}{1987}\natexlab{}.
\newblock \showarticletitle{Constructive Real Interpretation of Numerical
  Programs} \emph{(\bibinfo{series}{SIGPLAN '87})}. \bibinfo{publisher}{ACM},
  \bibinfo{pages}{214--221}.
\newblock
\showISBNx{0-89791-235-7}
\urldef\tempurl%
\url{https://doi.org/10.1145/29650.29673}
\showDOI{\tempurl}


\bibitem[\protect\citeauthoryear{Boehm}{Boehm}{2004}]%
        {boehm-java}
\bibfield{author}{\bibinfo{person}{Hans~J. Boehm}.}
  \bibinfo{year}{2004}\natexlab{}.
\newblock \showarticletitle{The constructive reals as a Java Library}.
\newblock \bibinfo{journal}{\emph{Journal Logical and Algebraic Programming}}
  \bibinfo{volume}{64} (\bibinfo{year}{2004}), \bibinfo{pages}{3--11}.
\newblock


\bibitem[\protect\citeauthoryear{Boehm}{Boehm}{2020}]%
        {api-for-real-numbers}
\bibfield{author}{\bibinfo{person}{Hans-J. Boehm}.}
  \bibinfo{year}{2020}\natexlab{}.
\newblock \showarticletitle{Towards an API for the Real Numbers}. In
  \bibinfo{booktitle}{\emph{Proceedings of the 41st ACM SIGPLAN Conference on
  Programming Language Design and Implementation}} (London, UK)
  \emph{(\bibinfo{series}{PLDI 2020})}. \bibinfo{publisher}{Association for
  Computing Machinery}, \bibinfo{address}{New York, NY, USA},
  \bibinfo{pages}{562–576}.
\newblock
\showISBNx{9781450376136}
\urldef\tempurl%
\url{https://doi.org/10.1145/3385412.3386037}
\showDOI{\tempurl}


\bibitem[\protect\citeauthoryear{Boehm, Cartwright, Riggle, and
  O'Donnell}{Boehm et~al\mbox{.}}{1986}]%
        {boehm-compare}
\bibfield{author}{\bibinfo{person}{Hans~J. Boehm}, \bibinfo{person}{Robert
  Cartwright}, \bibinfo{person}{Mark Riggle}, {and} \bibinfo{person}{Michael~J.
  O'Donnell}.} \bibinfo{year}{1986}\natexlab{}.
\newblock \showarticletitle{Exact Real Arithmetic: A Case Study in Higher Order
  Programming} \emph{(\bibinfo{series}{LFP '86})}. \bibinfo{publisher}{ACM},
  \bibinfo{pages}{162--173}.
\newblock
\showISBNx{0-89791-200-4}
\urldef\tempurl%
\url{https://doi.org/10.1145/319838.319860}
\showDOI{\tempurl}


\bibitem[\protect\citeauthoryear{Briggs}{Briggs}{2019}]%
        {gelpia}
\bibfield{author}{\bibinfo{person}{Ian Briggs}.}
  \bibinfo{year}{2019}\natexlab{}.
\newblock \bibinfo{title}{Global Extrema Locator Parallelization for Interval
  Arithmetic}.
\newblock
\newblock
\urldef\tempurl%
\url{https://github.com/soarlab/gelpia}
\showURL{%
\tempurl}


\bibitem[\protect\citeauthoryear{Briggs}{Briggs}{2006}]%
        {constructive-real-python}
\bibfield{author}{\bibinfo{person}{Keith Briggs}.}
  \bibinfo{year}{2006}\natexlab{}.
\newblock \showarticletitle{Implementing exact real arithmetic in python, C++
  and C}.
\newblock \bibinfo{journal}{\emph{Theoretical Computer Science}}
  \bibinfo{volume}{351}, \bibinfo{number}{1} (\bibinfo{year}{2006}),
  \bibinfo{pages}{74 -- 81}.
\newblock
\showISSN{0304-3975}
\urldef\tempurl%
\url{https://doi.org/10.1016/j.tcs.2005.09.058}
\showDOI{\tempurl}
\newblock
\shownote{Real Numbers and Computers.}


\bibitem[\protect\citeauthoryear{{Cherubin}, {Cattaneo}, {Chiari}, {Bello}, and
  {Agosta}}{{Cherubin} et~al\mbox{.}}{2020}]%
        {taffo}
\bibfield{author}{\bibinfo{person}{S. {Cherubin}}, \bibinfo{person}{D.
  {Cattaneo}}, \bibinfo{person}{M. {Chiari}}, \bibinfo{person}{A.~D. {Bello}},
  {and} \bibinfo{person}{G. {Agosta}}.} \bibinfo{year}{2020}\natexlab{}.
\newblock \showarticletitle{TAFFO: Tuning Assistant for Floating to Fixed Point
  Optimization}.
\newblock \bibinfo{journal}{\emph{IEEE Embedded Systems Letters}}
  \bibinfo{volume}{12}, \bibinfo{number}{1} (\bibinfo{year}{2020}),
  \bibinfo{pages}{5--8}.
\newblock


\bibitem[\protect\citeauthoryear{Chiang, Baranowski, Briggs, Solovyev,
  Gopalakrishnan, and Rakamari\'c}{Chiang et~al\mbox{.}}{2017}]%
        {fptuner}
\bibfield{author}{\bibinfo{person}{Wei-Fan Chiang}, \bibinfo{person}{Mark
  Baranowski}, \bibinfo{person}{Ian Briggs}, \bibinfo{person}{Alexey Solovyev},
  \bibinfo{person}{Ganesh Gopalakrishnan}, {and} \bibinfo{person}{Zvonimir
  Rakamari\'c}.} \bibinfo{year}{2017}\natexlab{}.
\newblock \showarticletitle{Rigorous Floating-point Mixed-precision Tuning}
  \emph{(\bibinfo{series}{POPL})}. \bibinfo{pages}{300--315}.
\newblock
\showISBNx{978-1-4503-4660-3}
\urldef\tempurl%
\url{https://doi.org/10.1145/3009837.3009846}
\showDOI{\tempurl}


\bibitem[\protect\citeauthoryear{Chiang, Gopalakrishnan, Rakamaric, and
  Solovyev}{Chiang et~al\mbox{.}}{2014}]%
        {s3fp}
\bibfield{author}{\bibinfo{person}{Wei-Fan Chiang}, \bibinfo{person}{Ganesh
  Gopalakrishnan}, \bibinfo{person}{Zvonimir Rakamaric}, {and}
  \bibinfo{person}{Alexey Solovyev}.} \bibinfo{year}{2014}\natexlab{}.
\newblock \showarticletitle{Efficient Search for Inputs Causing High
  Floating-Point Errors}. In \bibinfo{booktitle}{\emph{Proceedings of the 19th
  ACM SIGPLAN Symposium on Principles and Practice of Parallel Programming}}
  \emph{(\bibinfo{series}{PPoPP ’14})}. \bibinfo{publisher}{Association for
  Computing Machinery}, \bibinfo{address}{New York, NY, USA},
  \bibinfo{pages}{43–52}.
\newblock
\showISBNx{9781450326568}
\urldef\tempurl%
\url{https://doi.org/10.1145/2555243.2555265}
\showDOI{\tempurl}


\bibitem[\protect\citeauthoryear{Chowdhary, Lim, and Nagarakatte}{Chowdhary
  et~al\mbox{.}}{2020}]%
        {positdebug}
\bibfield{author}{\bibinfo{person}{Sangeeta Chowdhary}, \bibinfo{person}{Jay~P.
  Lim}, {and} \bibinfo{person}{Santosh Nagarakatte}.}
  \bibinfo{year}{2020}\natexlab{}.
\newblock \showarticletitle{Debugging and Detecting Numerical Errors in
  Computation with Posits}. In \bibinfo{booktitle}{\emph{Proceedings of the
  41st ACM SIGPLAN Conference on Programming Language Design and
  Implementation}} (London, UK) \emph{(\bibinfo{series}{PLDI 2020})}.
  \bibinfo{publisher}{Association for Computing Machinery},
  \bibinfo{address}{New York, NY, USA}, \bibinfo{pages}{731–746}.
\newblock
\showISBNx{9781450376136}
\urldef\tempurl%
\url{https://doi.org/10.1145/3385412.3386004}
\showDOI{\tempurl}


\bibitem[\protect\citeauthoryear{Damouche, Martel, and Chapoutot}{Damouche
  et~al\mbox{.}}{2015}]%
        {salsa-2}
\bibfield{author}{\bibinfo{person}{Nasrine Damouche}, \bibinfo{person}{Matthieu
  Martel}, {and} \bibinfo{person}{Alexandre Chapoutot}.}
  \bibinfo{year}{2015}\natexlab{}.
\newblock \showarticletitle{Formal Methods for Industrial Critical Systems:
  20th International Workshop, FMICS 2015 Oslo, Norway, June 22-23, 2015
  Proceedings}.
\newblock  (\bibinfo{year}{2015}), \bibinfo{pages}{31--46}.
\newblock
\showISBNx{978-3-319-19458-5}


\bibitem[\protect\citeauthoryear{Darulova and Kuncak}{Darulova and
  Kuncak}{2014}]%
        {rosa}
\bibfield{author}{\bibinfo{person}{Eva Darulova} {and} \bibinfo{person}{Viktor
  Kuncak}.} \bibinfo{year}{2014}\natexlab{}.
\newblock \showarticletitle{Sound Compilation of Reals}
  \emph{(\bibinfo{series}{POPL})}. \bibinfo{numpages}{14}~pages.
\newblock
\showISBNx{978-1-4503-2544-8}
\urldef\tempurl%
\url{http://doi.acm.org/10.1145/2535838.2535874}
\showURL{%
\tempurl}


\bibitem[\protect\citeauthoryear{Das, Briggs, Gopalakrishnan, Krishnamoorthy,
  and Panchekha}{Das et~al\mbox{.}}{2020}]%
        {satire}
\bibfield{author}{\bibinfo{person}{Arnab Das}, \bibinfo{person}{Ian Briggs},
  \bibinfo{person}{Ganesh Gopalakrishnan}, \bibinfo{person}{Sriram
  Krishnamoorthy}, {and} \bibinfo{person}{Pavel Panchekha}.}
  \bibinfo{year}{2020}\natexlab{}.
\newblock \showarticletitle{Scalable yet Rigorous Floating-Point Error
  Analysis}. In \bibinfo{booktitle}{\emph{2020 SC20: International Conference
  for High Performance Computing, Networking, Storage and Analysis (SC)}}.
  \bibinfo{publisher}{IEEE Computer Society}, \bibinfo{address}{Los Alamitos,
  CA, USA}, \bibinfo{pages}{1--14}.
\newblock
\urldef\tempurl%
\url{https://doi.org/10.1109/SC41405.2020.00055}
\showDOI{\tempurl}


\bibitem[\protect\citeauthoryear{{European Commission}}{{European
  Commission}}{1998}]%
        {euro-rounding}
\bibfield{author}{\bibinfo{person}{{European Commission}}.}
  \bibinfo{year}{1998}\natexlab{}.
\newblock \bibinfo{booktitle}{\emph{The introduction of the euro and the
  rounding of currency amounts}}.
\newblock \bibinfo{publisher}{European Commission, Directorate General II
  Economic and Financial Affairs}.
\newblock


\bibitem[\protect\citeauthoryear{Fousse, Hanrot, Lef\`evre, P\'elissier, and
  Zimmermann}{Fousse et~al\mbox{.}}{2007}]%
        {mpfr}
\bibfield{author}{\bibinfo{person}{Laurent Fousse}, \bibinfo{person}{Guillaume
  Hanrot}, \bibinfo{person}{Vincent Lef\`evre}, \bibinfo{person}{Patrick
  P\'elissier}, {and} \bibinfo{person}{Paul Zimmermann}.}
  \bibinfo{year}{2007}\natexlab{}.
\newblock \showarticletitle{{MPFR}: A Multiple-Precision Binary Floating-Point
  Library with Correct Rounding}.
\newblock \bibinfo{journal}{\emph{ACM Trans. Math. Software}}
  \bibinfo{volume}{33}, \bibinfo{number}{2} (\bibinfo{date}{June}
  \bibinfo{year}{2007}), \bibinfo{pages}{13:1--13:15}.
\newblock
\urldef\tempurl%
\url{http://doi.acm.org/10.1145/1236463.1236468}
\showURL{%
\tempurl}


\bibitem[\protect\citeauthoryear{Fu and Su}{Fu and Su}{2017}]%
        {coverme}
\bibfield{author}{\bibinfo{person}{Zhoulai Fu} {and} \bibinfo{person}{Zhendong
  Su}.} \bibinfo{year}{2017}\natexlab{}.
\newblock \showarticletitle{Achieving high coverage for floating-point code via
  unconstrained programming}.
\newblock \bibinfo{journal}{\emph{ACM SIGPLAN Notices}}  \bibinfo{volume}{52}
  (\bibinfo{date}{06} \bibinfo{year}{2017}), \bibinfo{pages}{306--319}.
\newblock
\urldef\tempurl%
\url{https://doi.org/10.1145/3140587.3062383}
\showDOI{\tempurl}


\bibitem[\protect\citeauthoryear{Goualard}{Goualard}{2019}]%
        {gaol}
\bibfield{author}{\bibinfo{person}{Frederic Goualard}.}
  \bibinfo{year}{2019}\natexlab{}.
\newblock \bibinfo{title}{Gaol: NOT Just Another Interval Library}.
\newblock
\newblock
\urldef\tempurl%
\url{https://sourceforge.net/projects/gaol/}
\showURL{%
\tempurl}


\bibitem[\protect\citeauthoryear{Granvilliers and Benhamou}{Granvilliers and
  Benhamou}{2006}]%
        {realpaver}
\bibfield{author}{\bibinfo{person}{Laurent Granvilliers} {and}
  \bibinfo{person}{Fr\'{e}d\'{e}ric Benhamou}.}
  \bibinfo{year}{2006}\natexlab{}.
\newblock \showarticletitle{Algorithm 852: RealPaver: An Interval Solver Using
  Constraint Satisfaction Techniques}.
\newblock \bibinfo{journal}{\emph{ACM Trans. Math. Softw.}}
  \bibinfo{volume}{32}, \bibinfo{number}{1} (\bibinfo{date}{March}
  \bibinfo{year}{2006}), \bibinfo{pages}{138–156}.
\newblock
\showISSN{0098-3500}
\urldef\tempurl%
\url{https://doi.org/10.1145/1132973.1132980}
\showDOI{\tempurl}


\bibitem[\protect\citeauthoryear{Hamming}{Hamming}{1987}]%
        {book87-nmse}
\bibfield{author}{\bibinfo{person}{Richard Hamming}.}
  \bibinfo{year}{1987}\natexlab{}.
\newblock \bibinfo{booktitle}{\emph{Numerical Methods for Scientists and
  Engineers} (\bibinfo{edition}{2nd} ed.)}.
\newblock \bibinfo{publisher}{Dover Publications}.
\newblock


\bibitem[\protect\citeauthoryear{Hui~Guo}{Hui~Guo}{2020}]%
        {fpgen}
\bibfield{author}{\bibinfo{person}{Cindy Rubio-{G}onz\'alez Hui~Guo}.}
  \bibinfo{year}{2020}\natexlab{}.
\newblock \showarticletitle{Efficient Generation of Error-Inducing
  Floating-Point Inputs via Symbolic Execution}
  \emph{(\bibinfo{series}{ICSE'20})}.
\newblock


\bibitem[\protect\citeauthoryear{Izycheva and Darulova}{Izycheva and
  Darulova}{2017}]%
        {daisy}
\bibfield{author}{\bibinfo{person}{Anastasiia Izycheva} {and}
  \bibinfo{person}{Eva Darulova}.} \bibinfo{year}{2017}\natexlab{}.
\newblock \showarticletitle{On sound relative error bounds for floating-point
  arithmetic} \emph{(\bibinfo{series}{{FMCAD}})}. \bibinfo{pages}{15--22}.
\newblock
\urldef\tempurl%
\url{https://doi.org/10.23919/FMCAD.2017.8102236}
\showDOI{\tempurl}


\bibitem[\protect\citeauthoryear{Jaulin, Kieffer, Didrit, and Walter}{Jaulin
  et~al\mbox{.}}{2001}]%
        {applied-interval-analysis}
\bibfield{author}{\bibinfo{person}{L. Jaulin}, \bibinfo{person}{M. Kieffer},
  \bibinfo{person}{O. Didrit}, {and} \bibinfo{person}{E. Walter}.}
  \bibinfo{year}{2001}\natexlab{}.
\newblock \bibinfo{booktitle}{\emph{Applied Interval Analysis: With Examples in
  Parameter and State Estimation, Robust Control and Robotics}}.
\newblock \bibinfo{publisher}{Springer London}.
\newblock
\showISBNx{9781852332198}
\showLCCN{2001020164}
\urldef\tempurl%
\url{https://books.google.com/books?id=ZG0qXkYUe\_AC}
\showURL{%
\tempurl}


\bibitem[\protect\citeauthoryear{Kahan}{Kahan}{2000}]%
        {berkeley00-needle-like}
\bibfield{author}{\bibinfo{person}{William Kahan}.}
  \bibinfo{year}{2000}\natexlab{}.
\newblock \bibinfo{booktitle}{\emph{Miscalculating Area and Angles of a
  Needle-like Triangle}}.
\newblock \bibinfo{type}{{T}echnical {R}eport}.
  \bibinfo{institution}{University of California, Berkeley}.
\newblock
\urldef\tempurl%
\url{http://www.cs.berkeley.edu/~wkahan/Triangle.pdf}
\showURL{%
\tempurl}


\bibitem[\protect\citeauthoryear{Kahan and Darcy}{Kahan and Darcy}{1998}]%
        {kahan-java-hurts}
\bibfield{author}{\bibinfo{person}{W. Kahan} {and} \bibinfo{person}{Joseph~D.
  Darcy}.} \bibinfo{year}{1998}\natexlab{}.
\newblock \bibinfo{booktitle}{\emph{How {Java}'s Floating-Point Hurts Everyone
  Everywhere}}.
\newblock \bibinfo{type}{Technical Report}. \bibinfo{institution}{University of
  California, Berkeley}. \bibinfo{pages}{80} pages.
\newblock
\urldef\tempurl%
\url{http://www.cs.berkeley.edu/~wkahan/JAVAhurt.pdf}
\showURL{%
\tempurl}


\bibitem[\protect\citeauthoryear{Lee and Boehm}{Lee and Boehm}{1990a}]%
        {interval-constructive}
\bibfield{author}{\bibinfo{person}{Vernon~A. Lee} {and}
  \bibinfo{person}{Hans-J. Boehm}.} \bibinfo{year}{1990}\natexlab{a}.
\newblock \showarticletitle{Optimizing Programs over the Constructive Reals}
  \emph{(\bibinfo{series}{PLDI ’90})}. \bibinfo{address}{New York, NY, USA}.
\newblock
\urldef\tempurl%
\url{https://doi.org/10.1145/93542.93558}
\showDOI{\tempurl}


\bibitem[\protect\citeauthoryear{Lee and Boehm}{Lee and Boehm}{1990b}]%
        {boehm-fast}
\bibfield{author}{\bibinfo{person}{Vernon~A. Lee, Jr.} {and}
  \bibinfo{person}{Hans~J. Boehm}.} \bibinfo{year}{1990}\natexlab{b}.
\newblock \showarticletitle{Optimizing Programs over the Constructive Reals}
  \emph{(\bibinfo{series}{PLDI '90})}. \bibinfo{publisher}{ACM},
  \bibinfo{pages}{102--111}.
\newblock
\showISBNx{0-89791-364-7}
\urldef\tempurl%
\url{https://doi.org/10.1145/93542.93558}
\showDOI{\tempurl}


\bibitem[\protect\citeauthoryear{Li and Yong}{Li and Yong}{2007}]%
        {constructive-real-tweaks}
\bibfield{author}{\bibinfo{person}{Yong Li} {and} \bibinfo{person}{Jun-Hai
  Yong}.} \bibinfo{year}{2007}\natexlab{}.
\newblock \showarticletitle{{Efficient Exact Arithmetic over Constructive
  Reals}}. In \bibinfo{booktitle}{\emph{{The 4th Annual Conference on Theory
  and Applications of Models of Computation}}}. \bibinfo{address}{Shanghai,
  China}.
\newblock
\urldef\tempurl%
\url{https://hal.inria.fr/inria-00517598}
\showURL{%
\tempurl}


\bibitem[\protect\citeauthoryear{Magron, Constantinides, and Donaldson}{Magron
  et~al\mbox{.}}{2017}]%
        {real2float}
\bibfield{author}{\bibinfo{person}{Victor Magron}, \bibinfo{person}{George
  Constantinides}, {and} \bibinfo{person}{Alastair Donaldson}.}
  \bibinfo{year}{2017}\natexlab{}.
\newblock \showarticletitle{Certified roundoff error bounds using semidefinite
  programming}.
\newblock \bibinfo{journal}{\emph{Transactions on Mathematical Software}}
  \bibinfo{volume}{43}, \bibinfo{number}{4} (\bibinfo{year}{2017}),
  \bibinfo{pages}{34}.
\newblock


\bibitem[\protect\citeauthoryear{Martel}{Martel}{2009}]%
        {salsa-1}
\bibfield{author}{\bibinfo{person}{Matthieu Martel}.}
  \bibinfo{year}{2009}\natexlab{}.
\newblock \showarticletitle{Program Transformation for Numerical Precision}
  \emph{(\bibinfo{series}{PEPM '09})}. \bibinfo{publisher}{ACM},
  \bibinfo{address}{New York, NY, USA}, \bibinfo{pages}{101--110}.
\newblock
\showISBNx{978-1-60558-327-3}
\urldef\tempurl%
\url{http://doi.acm.org/10.1145/1480945.1480960}
\showURL{%
\tempurl}


\bibitem[\protect\citeauthoryear{Martel}{Martel}{2018}]%
        {martel-types}
\bibfield{author}{\bibinfo{person}{Matthieu Martel}.}
  \bibinfo{year}{2018}\natexlab{}.
\newblock \showarticletitle{Strongly Typed Numerical Computations}.
\newblock \bibinfo{journal}{\emph{ICFEM}} (\bibinfo{year}{2018}),
  \bibinfo{pages}{197--214}.
\newblock


\bibitem[\protect\citeauthoryear{Mascarenhas}{Mascarenhas}{2016}]%
        {moore-ivals}
\bibfield{author}{\bibinfo{person}{Walter~F. Mascarenhas}.}
  \bibinfo{year}{2016}\natexlab{}.
\newblock \showarticletitle{Moore: Interval Arithmetic in Modern {C++}}.
\newblock \bibinfo{journal}{\emph{CoRR}}  \bibinfo{volume}{abs/1611.09567}
  (\bibinfo{year}{2016}).
\newblock
\showeprint[arxiv]{1611.09567}
\urldef\tempurl%
\url{http://arxiv.org/abs/1611.09567}
\showURL{%
\tempurl}


\bibitem[\protect\citeauthoryear{McCullough and Vinod}{McCullough and
  Vinod}{1999}]%
        {distort-stock}
\bibfield{author}{\bibinfo{person}{B.~D. McCullough} {and}
  \bibinfo{person}{H.~D. Vinod}.} \bibinfo{year}{1999}\natexlab{}.
\newblock \showarticletitle{The Numerical Reliability of Econometric Software}.
\newblock \bibinfo{journal}{\emph{Journal of Economic Literature}}
  \bibinfo{volume}{37}, \bibinfo{number}{2} (\bibinfo{year}{1999}),
  \bibinfo{pages}{633--665}.
\newblock


\bibitem[\protect\citeauthoryear{Melquiond}{Melquiond}{2006}]%
        {boost-ivals}
\bibfield{author}{\bibinfo{person}{Guillaume Melquiond}.}
  \bibinfo{year}{2006}\natexlab{}.
\newblock \bibinfo{title}{Boost Interval Arithmetic Library}.
\newblock
\newblock
\urldef\tempurl%
\url{https://www.boost.org/doc/libs/1_73_0/libs/numeric/interval/doc/interval.htm}
\showURL{%
\tempurl}


\bibitem[\protect\citeauthoryear{Minsky}{Minsky}{1967}]%
        {computable-reals}
\bibfield{author}{\bibinfo{person}{Marvin~L. Minsky}.}
  \bibinfo{year}{1967}\natexlab{}.
\newblock \bibinfo{booktitle}{\emph{Computation: Finite and Infinite
  Machines}}.
\newblock \bibinfo{publisher}{Prentice-Hall, Inc.}
\newblock
\showISBNx{0-13-165563-9}


\bibitem[\protect\citeauthoryear{Ménissier-Morain}{Ménissier-Morain}{2005}]%
        {constructive-real-verified}
\bibfield{author}{\bibinfo{person}{Valérie Ménissier-Morain}.}
  \bibinfo{year}{2005}\natexlab{}.
\newblock \showarticletitle{Arbitrary precision real arithmetic: design and
  algorithms}.
\newblock \bibinfo{journal}{\emph{The Journal of Logic and Algebraic
  Programming}} \bibinfo{volume}{64}, \bibinfo{number}{1}
  (\bibinfo{year}{2005}), \bibinfo{pages}{13 -- 39}.
\newblock
\showISSN{1567-8326}
\urldef\tempurl%
\url{https://doi.org/10.1016/j.jlap.2004.07.003}
\showDOI{\tempurl}
\newblock
\shownote{Practical development of exact real number computation.}


\bibitem[\protect\citeauthoryear{Panchekha, Sanchez-Stern, Wilcox, and
  Tatlock}{Panchekha et~al\mbox{.}}{2015}]%
        {herbie}
\bibfield{author}{\bibinfo{person}{Pavel Panchekha}, \bibinfo{person}{Alex
  Sanchez-Stern}, \bibinfo{person}{James~R. Wilcox}, {and}
  \bibinfo{person}{Zachary Tatlock}.} \bibinfo{year}{2015}\natexlab{}.
\newblock \showarticletitle{Automatically Improving Accuracy for Floating Point
  Expressions} \emph{(\bibinfo{series}{PLDI})}.
\newblock


\bibitem[\protect\citeauthoryear{Revol and Rouillier}{Revol and
  Rouillier}{2005}]%
        {mpfi}
\bibfield{author}{\bibinfo{person}{Nathalie Revol} {and}
  \bibinfo{person}{Fabrice Rouillier}.} \bibinfo{year}{2005}\natexlab{}.
\newblock \showarticletitle{Motivations for an Arbitrary Precision Interval
  Arithmetic and the MPFI Library}.
\newblock \bibinfo{journal}{\emph{Reliable Computing}} \bibinfo{volume}{11},
  \bibinfo{number}{4} (\bibinfo{year}{2005}), \bibinfo{pages}{275--290}.
\newblock
\urldef\tempurl%
\url{https://doi.org/10.1007/s11155-005-6891-y}
\showDOI{\tempurl}


\bibitem[\protect\citeauthoryear{Rubio-Gonz{\'a}lez~\textit{et
  al.}}{Rubio-Gonz{\'a}lez~\textit{et al.}}{2013}]%
        {precimonious}
\bibfield{author}{\bibinfo{person}{Cindy Rubio-Gonz{\'a}lez~\textit{et al.}}}
  \bibinfo{year}{2013}\natexlab{}.
\newblock \showarticletitle{Precimonious: Tuning assistant for floating-point
  precision} \emph{(\bibinfo{series}{{SC}})}. IEEE, \bibinfo{pages}{1--12}.
\newblock


\bibitem[\protect\citeauthoryear{Sanchez-Stern, Panchekha, Lerner, and
  Tatlock}{Sanchez-Stern et~al\mbox{.}}{2018}]%
        {herbgrind}
\bibfield{author}{\bibinfo{person}{Alex Sanchez-Stern}, \bibinfo{person}{Pavel
  Panchekha}, \bibinfo{person}{Sorin Lerner}, {and} \bibinfo{person}{Zachary
  Tatlock}.} \bibinfo{year}{2018}\natexlab{}.
\newblock \showarticletitle{Finding Root Causes of Floating Point Error}
  \emph{(\bibinfo{series}{{PLDI}})}. \bibinfo{pages}{256--269}.
\newblock
\urldef\tempurl%
\url{https://doi.org/10.1145/3192366.3192411}
\showDOI{\tempurl}


\bibitem[\protect\citeauthoryear{Solovyev, Jacobsen, Rakamaric, and
  Gopalakrishnan}{Solovyev et~al\mbox{.}}{2015}]%
        {fptaylor}
\bibfield{author}{\bibinfo{person}{Alexey Solovyev}, \bibinfo{person}{Charlie
  Jacobsen}, \bibinfo{person}{Zvonimir Rakamaric}, {and}
  \bibinfo{person}{Ganesh Gopalakrishnan}.} \bibinfo{year}{2015}\natexlab{}.
\newblock \showarticletitle{Rigorous Estimation of Floating-Point Round-off
  Errors with Symbolic Taylor Expansions} \emph{(\bibinfo{series}{FM})}.
\newblock


\bibitem[\protect\citeauthoryear{Tang, Barr, Li, and Su}{Tang
  et~al\mbox{.}}{2010}]%
        {perturbing-numerical}
\bibfield{author}{\bibinfo{person}{Enyi Tang}, \bibinfo{person}{Earl Barr},
  \bibinfo{person}{Xuandong Li}, {and} \bibinfo{person}{Zhendong Su}.}
  \bibinfo{year}{2010}\natexlab{}.
\newblock \showarticletitle{Perturbing Numerical Calculations for Statistical
  Analysis of Floating-Point Program (in)Stability}. In
  \bibinfo{booktitle}{\emph{Proceedings of the 19th International Symposium on
  Software Testing and Analysis}} (Trento, Italy) \emph{(\bibinfo{series}{ISSTA
  ’10})}. \bibinfo{publisher}{Association for Computing Machinery},
  \bibinfo{address}{New York, NY, USA}, \bibinfo{pages}{131–142}.
\newblock
\showISBNx{9781605588230}
\urldef\tempurl%
\url{https://doi.org/10.1145/1831708.1831724}
\showDOI{\tempurl}


\bibitem[\protect\citeauthoryear{Titolo, Feli{\'u}, Moscato, and Munoz}{Titolo
  et~al\mbox{.}}{2018}]%
        {precisa}
\bibfield{author}{\bibinfo{person}{Laura Titolo}, \bibinfo{person}{Marco~A
  Feli{\'u}}, \bibinfo{person}{Mariano Moscato}, {and}
  \bibinfo{person}{C{\'e}sar~A Munoz}.} \bibinfo{year}{2018}\natexlab{}.
\newblock \showarticletitle{An Abstract Interpretation Framework for the
  Round-Off Error Analysis of Floating-Point Programs}
  \emph{(\bibinfo{series}{{VMCAI}})}. \bibinfo{pages}{516--537}.
\newblock


\bibitem[\protect\citeauthoryear{Toronto and McCarthy}{Toronto and
  McCarthy}{2014}]%
        {cse14-practical-fp}
\bibfield{author}{\bibinfo{person}{N. Toronto} {and} \bibinfo{person}{J.
  McCarthy}.} \bibinfo{year}{2014}\natexlab{}.
\newblock \showarticletitle{Practically Accurate Floating-Point Math}.
\newblock \bibinfo{journal}{\emph{Computing in Science Engineering}}
  \bibinfo{volume}{16}, \bibinfo{number}{4} (\bibinfo{date}{July}
  \bibinfo{year}{2014}), \bibinfo{pages}{80--95}.
\newblock


\bibitem[\protect\citeauthoryear{{U.S. General Accounting Office}}{{U.S.
  General Accounting Office}}{1992}]%
        {patriot}
\bibfield{author}{\bibinfo{person}{{U.S. General Accounting Office}}.}
  \bibinfo{year}{1992}\natexlab{}.
\newblock \bibinfo{title}{Patriot Missile Defense: Software Problem Led to
  System Failure at Dhahran, Saudi Arabia}.
\newblock
\newblock
\urldef\tempurl%
\url{http://www.gao.gov/products/IMTEC-92-26}
\showURL{%
\tempurl}


\bibitem[\protect\citeauthoryear{Weber-Wulff}{Weber-Wulff}{1992}]%
        {round-elections}
\bibfield{author}{\bibinfo{person}{Debora Weber-Wulff}.}
  \bibinfo{year}{1992}\natexlab{}.
\newblock \bibinfo{title}{Rounding error changes Parliament makeup}.
\newblock
\newblock
\urldef\tempurl%
\url{http://catless.ncl.ac.uk/Risks/13.37.html\#subj4}
\showURL{%
\tempurl}


\bibitem[\protect\citeauthoryear{Wolfram}{Wolfram}{2020}]%
        {mathematica-maxextraprecision}
\bibfield{author}{\bibinfo{person}{Wolfram}.} \bibinfo{year}{2020}\natexlab{}.
\newblock \bibinfo{title}{\$MaxExtraPrecision---Wolfram Language
  Documentation}.
\newblock
\newblock
\urldef\tempurl%
\url{https://reference.wolfram.com/language/ref/\$MaxExtraPrecision.html}
\showURL{%
\tempurl}


\bibitem[\protect\citeauthoryear{{WR-SWT}}{{WR-SWT}}{2020}]%
        {xsc}
\bibfield{author}{\bibinfo{person}{{WR-SWT}}.} \bibinfo{year}{2020}\natexlab{}.
\newblock \bibinfo{title}{History of {XSC}-{L}anguages and Credits}.
\newblock
\newblock
\urldef\tempurl%
\url{http://www.xsc.de}
\showURL{%
\tempurl}


\bibitem[\protect\citeauthoryear{{Yi}, {Chen}, {Mao}, and {Ji}}{{Yi}
  et~al\mbox{.}}{2017}]%
        {eagt}
\bibfield{author}{\bibinfo{person}{X. {Yi}}, \bibinfo{person}{L. {Chen}},
  \bibinfo{person}{X. {Mao}}, {and} \bibinfo{person}{T. {Ji}}.}
  \bibinfo{year}{2017}\natexlab{}.
\newblock \showarticletitle{Efficient Global Search for Inputs Triggering High
  Floating-Point Inaccuracies}. In \bibinfo{booktitle}{\emph{2017 24th
  Asia-Pacific Software Engineering Conference (APSEC)}}.
  \bibinfo{pages}{11--20}.
\newblock


\bibitem[\protect\citeauthoryear{Yi, Chen, Mao, and Ji}{Yi
  et~al\mbox{.}}{2019}]%
        {autornp}
\bibfield{author}{\bibinfo{person}{Xin Yi}, \bibinfo{person}{Liqian Chen},
  \bibinfo{person}{Xiaoguang Mao}, {and} \bibinfo{person}{Tao Ji}.}
  \bibinfo{year}{2019}\natexlab{}.
\newblock \showarticletitle{Efficient Automated Repair of High Floating-Point
  Errors in Numerical Libraries}.
\newblock \bibinfo{journal}{\emph{Proc. ACM Program. Lang.}}
  \bibinfo{volume}{3}, \bibinfo{number}{POPL}, Article \bibinfo{articleno}{56}
  (\bibinfo{date}{Jan.} \bibinfo{year}{2019}), \bibinfo{numpages}{29}~pages.
\newblock
\urldef\tempurl%
\url{https://doi.org/10.1145/3290369}
\showDOI{\tempurl}


\bibitem[\protect\citeauthoryear{{Zou}, {Wang}, {Xiong}, {Zhang}, {Su}, and
  {Mei}}{{Zou} et~al\mbox{.}}{2015}]%
        {lsga}
\bibfield{author}{\bibinfo{person}{D. {Zou}}, \bibinfo{person}{R. {Wang}},
  \bibinfo{person}{Y. {Xiong}}, \bibinfo{person}{L. {Zhang}},
  \bibinfo{person}{Z. {Su}}, {and} \bibinfo{person}{H. {Mei}}.}
  \bibinfo{year}{2015}\natexlab{}.
\newblock \showarticletitle{A Genetic Algorithm for Detecting Significant
  Floating-Point Inaccuracies}. In \bibinfo{booktitle}{\emph{2015 IEEE/ACM 37th
  IEEE International Conference on Software Engineering}},
  Vol.~\bibinfo{volume}{1}. \bibinfo{pages}{529--539}.
\newblock


\end{thebibliography}

\end{document}